\documentclass[11pt]{amsart}
\usepackage[utf8]{inputenc}
\usepackage{amsmath,mathtools}
\usepackage{stmaryrd}
\usepackage{MnSymbol}
\usepackage{pst-plot} 
\usepackage{auto-pst-pdf}
\usepackage{tikz}
\usetikzlibrary{arrows,automata}
\usepackage{enumitem}
\usepackage[colorlinks]{hyperref}
\usepackage{listings}
\lstdefinelanguage{GAP}{%
  morekeywords={%
    Assert,Info,IsBound,QUIT,%
    TryNextMethod,Unbind,and,break,%
    continue,do,elif,%
    else,end,false,fi,for,%
    function,if,in,local,%
    mod,not,od,or,%
    quit,rec,repeat,return,%
    then,true,until,while%
  },%
  sensitive,%
  morecomment=[l]\#,%
  morestring=[b]",%
  morestring=[b]',%
}[keywords,comments,strings]

\usepackage[T1]{fontenc}
\usepackage{xcolor}
\newtheorem{thm}{Theorem}[section]
\newtheorem{cor}[thm]{Corollary}
\newtheorem{lem}[thm]{Lemma}
\newtheorem{prop}[thm]{Proposition}

\newtheorem{opm}[thm]{Open Problem}
\newcommand{\abs}[1]{\left\vert#1\right\vert}
\newcommand{\C}{${\mathcal L}$-representation}
\newcommand{\D}{minimal non-nilpotent}
\newcommand{\M}{NMN_{4}}

\def\pv#1{\ensuremath{\mathsf{#1}}}
\newcommand{\St}{\mathop{\mathrm{St}}\nolimits}
\begin{document}
\title{The rank of variants of nilpotent pseudovarieties}
\author{J. Almeida and M. H. Shahzamanian}
\address{J. Almeida and M. H. Shahzamanian\\ Centro de Matemática e Departamento de Matemática, Faculdade de Ciências,
Universidade do Porto, Rua do Campo Alegre, 687, 4169-007 Porto,
Portugal}
\email{jalmeida@fc.up.pt; m.h.shahzamanian@fc.up.pt}

\begin{abstract}
  We investigate the rank of pseudovarieties defined by
  several of the variants of nilpotency conditions for semigroups in
  the sense of Mal'cev. For several of them, we provide finite bases
  of pseudoidentities. We also show that the Neumann--Taylor variant
  does not have finite rank.
\end{abstract}
\maketitle


\section{Introduction}\label{pre}

A pseudovariety of semigroups is a class of finite semigroups closed under taking subsemigroups, homomorphic images, and finite direct products. 
By a theorem of Reiterman~\cite{Rei}, we may define a pseudovariety of semigroups by a set of pseudoidentities, which is called a basis. We say that a pseudovariety of semigroups is finitely based if it admits a finite basis of pseudoidentities. 

Let $\mathsf{V}$ be a semigroup pseudovariety.
We say that a semigroup $S$ is $n$-generated if there is some subset of~$S$ with at most $n$~elements that generates~$S$.
We say that $\mathsf{V}$ has rank at most $n$ if, whenever a semigroup $S$ is such that all its $n$-generated subsemigroups belong to~$\mathsf{V}$, so does $S$.
If there is such an integer $n\ge0$, then $\mathsf{V}$ is said to have finite rank and the minimum value of $n$ for which $\mathsf{V}$ has
rank at most $n$ is called the rank of~$\mathsf{V}$ and is denoted by
$\mbox{rk}(\mathsf{V})$. If there is no such $n$, then $\mathsf{V}$ is
said to have infinite rank.
The rank of $\mathsf{V}$ is less than or equal to the number of
variables involved in the ultimate description of $\mathsf{V}$ by
sequences of identities (in the sense of Eilenberg and
Sch{\"u}tzenberger \cite{Eil-Sch}) or in the description by pseudoidentities. 
Hence, a semigroup pseudovariety with infinite rank is non-finitely based. 
The class $\mathsf{V}^{(n)}$ is defined as the collection of all finite semigroups $S$ such that every $n$-generated subsemigroup of $S$ belongs to $\mathsf{V}$.  Note that the class $\mathsf{V}^{(n)}$ is a semigroup pseudovariety and if $\mathsf{V} =\mathsf{V}^{(n)}$ for some $n$ then $\mbox{rk}(\mathsf{V})\leq n$.

Mal'cev \cite{Mal} and independently Neumann and Taylor \cite{Neu-Tay} have shown that nilpotent groups may be defined by semigroup identities (that is, without using inverses). This leads to the notion of a nilpotent semigroup (in the sense of Mal'cev). 

For a semigroup $S$ and elements $x,y,z_{1},z_{2}, \ldots$ one recursively defines two sequences 
$$\lambda_n=\lambda_{n}(x,y,z_{1},\ldots, z_{n})\quad{\rm and} \quad \rho_n=\rho_{n}(x,y,z_{1},\ldots, z_{n})$$ by
$$\lambda_{0}=x, \quad \rho_{0}=y$$ and
$$\lambda_{n+1}=\lambda_{n} z_{n+1} \rho_{n}, \quad \rho_{n+1}=\rho_{n} z_{n+1} \lambda_{n}.$$
We use the notation $S^1$ where if $S$ has an identity element then $S^1=S$, otherwise $S^{1}=S\cup \{1\}$ is obtained by adding an identity element to $S$.
A semigroup is said to be \emph{nilpotent}\footnote{Mal'cev's original
  paper \cite{Mal} defines nilpotent semigroups in terms of a weaker
  property since Mal'cev does not allow $c_1,\ldots,c_n$ to take the
  value $1$ if $S\neq S^{1}$. The definition used here comes from
  \cite{Lal}. By \cite[Lemma 3.2]{Alm-Kuf-Sha}, the two definitions
  agree on the class of finite semigroups.}
(MN) if 
\begin{displaymath}
  \exists n\ge0\ \forall a,b\in S\ \forall c_{1}, \ldots , c_{n}\in S^{1},\ 
  \lambda_{n}(a,b,c_{1},\ldots, c_{n}) = \rho_{n}(a,b,c_{1},\ldots,c_{n}).
\end{displaymath}

Recall that a semigroup $S$ is said to be \emph{Neumann--Taylor}
\cite{Neu-Tay} (NT)  if 
\begin{displaymath}
  \exists n\ge2\ \forall a,b \in S\ \forall c_2,\ldots,c_n \in S^1,\  
  \lambda_n(a,b,1,c_2,\ldots,c_n)=\rho_n(a,b,1,c_2,\ldots,c_n)
\end{displaymath}
A semigroup $S$ is said to be \emph{positively Engel} (PE) if 
\begin{align*}
  &\exists n\ge2\ \forall a,b \in S\ \forall c\in S^{1},\\
  &\quad
    \lambda_{n}(a,b,1,1,c,c^{2},\ldots ,c^{n-2})
    =\rho_{n} (a,b,1,1,c,c^{2},\ldots ,c^{n-2}),
\end{align*}
while $S$ is said to be \emph{Thue--Morse} (TM) if 
\begin{displaymath}
  \exists n\ge0\ \forall a,b \in S,\ 
  \lambda_{n}(a,b,1,1,\ldots,1)=\rho_{n}(a,b,1,1,\ldots,1).
\end{displaymath}

Each of the above classes of finite semigroups (nilpotent, Neumann--Taylor, positively Engel, Thue--Morse) constitutes a pseudovariety. Actually, the above descriptions are examples of ultimate equational definitions of pseudovarieties in the sense of Eilenberg and Sch{\"u}tzenberger \cite{Eil-Sch}. We denote them respectively by $\mathsf{MN}$, $\mathsf{NT}$, $\mathsf{PE}$ and $\mathsf{TM}$. 
 
In this paper, we investigate the rank of pseudovarieties defined by Mal’cev nilpotency conditions. In particular, we show that the pseudovariety $\mathsf{NT}$ has infinite rank and, therefore, it is non-finitely based. 

We denote by $\mathsf{G_{nil}}$ the pseudovariety of all finite
nilpotent groups and by $\mathsf{\overline{G}_{nil}}$ the
pseudovariety of all finite semigroups whose subgroups belong to
$\mathsf{G_{nil}}$. The pseudovariety $\mathsf{BG}$ is the collection
of all finite block groups, that is, finite semigroups in which each
element has at most one inverse and the pseudovariety
$\mathsf{BG_{nil}}$ is the collection of all finite block groups whose
subgroups are nilpotent. Every finite semigroup that is not nilpotent
has a subsemigroup that is minimal for not being nilpotent, in the
sense that every proper subsemigroup and every Rees factor semigroup
is nilpotent. Semigroups with this minimality condition have been
described in \cite{Jes-Sha3} where they are called \D\ semigroups. It
was shown that a \D\ semigroup in the pseudovariety
$\mathsf{\overline{G}_{nil}}$ is of one of four types of semigroups.
These four types of semigroups are a right or left zero semigroup or
the union of a completely $0$-simple inverse ideal with a
$2$-generated subsemigroup. A minimal non-nilpotent semigroup $S$
which is not a group or a right or left zero semigroup has a
completely $0$-simple inverse ideal $M$ and $S$ acts on the ${\mathcal
  R}$-classes of $M$ where the different types of orbits of this
action determine the type of $S$. The pseudovariety
$\mathsf{BG_{nil}}$ does not contain any semigroup of the first type
but may have semigroups of other types, the pseudovariety
$\mathsf{PE}$ does not contain any semigroup of the first and second
types but may have semigroups of other types and the pseudovariety
$\mathsf{MN}$ does not contain any semigroup of any types. In this
paper, we define a pseudovariety by excluding non-nilpotent groups and
the first three types of minimal non-nilpotent semigroups, while
allowing the fourth type. We name this pseudovariety $\mathsf{\M}$.
Thus, it is a class of semigroups that is just one step away from
being $\mathsf{MN}$ while being contained in the pseudovariety
$\mathsf{PE}$; it has deserved our attention and we dedicate a
section to it.
It turns out that $\mathsf{\M}$ sits strictly between the
pseudovarieties $\mathsf{NT}$ and $\mathsf{PE}$.

We calculate the ranks of the pseudovarieties $\mathsf{MN}$,
$\mathsf{\M}$, $\mathsf{PE}$ and $\mathsf{TM}$. They are respectively
$4$, $3$, $2$ and $2$.

Let $\mathcal{C}$ be a class of finite semigroups. The pseudovariety
generated by $\mathcal{C}$ is denoted by $\langle \mathcal{C}
\rangle$. Let $\mathsf{Inv}$ be the class of all finite inverse
semigroups and $\mathsf{A}$ be the pseudovariety of all finite aperiodic
semigroups. At the end of the paper, we compare the pseudovarieties
$\langle\mathsf{Inv}\cap\mathsf{A}\rangle$ and
$\langle\mathsf{Inv}\rangle\cap \mathsf{A}$
with the pseudovariety $\mathsf{MN}\cap \mathsf{A}$. We show that the
pseudovariety $\langle\mathsf{Inv}\cap\mathsf{A}\rangle$ is strictly
contained in $\mathsf{MN}\cap\mathsf{A}$ and the pseudovarieties
$\langle\mathsf{Inv}\rangle\cap\mathsf{A}$ and $\mathsf{MN}\cap
\mathsf{A}$ are incomparable. We further compare the pseudovarieties
$\mathsf{MN}$, $\mathsf{NT}$, $\mathsf{\M}$, $\mathsf{PE}$,
$\mathsf{TM}$, $\mathsf{MN}^{(2)}$, and $\mathsf{MN}^{(3)}$. The
diagrams at Figure~\ref{fig1} represent the strict inclusion
relationships and equalities between these pseudovarieties
respectively in the general case and in the aperiodic case.

\begin{figure}\label{fig1}
    \centering
\psscalebox{1.0 1.0} 
{
\begin{pspicture}(0,-3.4185936)(25.18,3.4185936)
  \rput[bl](7.05,2.5005078){$\mathsf{PE}\cap\mathsf{A}
    =\mathsf{BG}\cap\mathsf{A}$}
\rput[bl](5.26,0.3005078){$\mathsf{MN}^{(3)}\cap\mathsf{A}$}
\rput[bl](5.26,1.3805078){$\mathsf{MN}^{(2)}\cap\mathsf{A}$}
\rput[bl](9.56,0.3005078){$\mathsf{NT}\cap\mathsf{A}$}
\rput[bl](9.56,1.3805078){$\mathsf{\M}\cap\mathsf{A}$}
\rput[bl](7.56,-0.8994922){$\mathsf{MN}\cap\mathsf{A}$}
\psline[linecolor=black, linewidth=0.04](10.06,1.3005078)(10.06,0.7005078)
\psline[linecolor=black, linewidth=0.04](8.46,2.400508)(9.96,1.8005078)
\psline[linecolor=black, linewidth=0.04](9.96,0.20050782)(8.46,-0.4994922)
\rput[bl](7.56,-1.7994922){$\langle\mathsf{Inv}\cap\mathsf{A}\rangle$}
\psline[linecolor=black, linewidth=0.04](8.36,-0.9594922)(8.36,-1.4194922)
\psline[linecolor=black, linewidth=0.04](8.26,2.400508)(6.7,1.8005078)
\psline[linecolor=black, linewidth=0.04](6.7,0.20050782)(8.26,-0.4994922)
\psline[linecolor=black, linewidth=0.04](6.3,1.3005078)(6.3,0.7005078)
\psline[linecolor=black, linewidth=0.04](9.5,1.5)(7,0.5)
\rput[bl](0.4,-1.0794922){$\mathsf{MN}^{(2)}$}
\rput[bl](0.55,-2.1394923){$\mathsf{MN}^{(3)}$}
\rput[bl](4.2,-2.1394923){$\mathsf{NT}$}
\rput[bl](3.9,-1.0794922){$\mathsf{\M}$}
\rput[bl](2.5,-3.339492){$\mathsf{MN}$}
\psline[linecolor=black, linewidth=0.04](4.5,-1.1394922)(4.5,-1.7394922)
\psline[linecolor=black, linewidth=0.04](2.9,-0.039492186)(4.38,-0.6694922)
\psline[linecolor=black, linewidth=0.04](4.4,-2.2394922)(2.9,-3)
\psline[linecolor=black, linewidth=0.04](2.7,-0.039492186)(1.4,-.74)
\psline[linecolor=black, linewidth=0.04](1.2,-2.2)(2.7,-3)
\psline[linecolor=black, linewidth=0.04](.8,-1.2)(.8,-1.8)
\psline[linecolor=black, linewidth=0.04](1.7,-2)(3.8,-.9)
\rput[bl](2.5,0.16050781){$\mathsf{PE}$}
\rput[bl](2.5,3.1405077){$\mathsf{BG}$}
\rput[bl](2.52,2.160508){$\mathsf{TM}$}
\rput[bl](2.52,1.1605078){$\mathsf{BG_{nil}}$}
\psline[linecolor=black, linewidth=0.04](2.8,1.0605078)(2.8,0.5505078)
\psline[linecolor=black, linewidth=0.04](2.8,3.0605078)(2.8,2.5605078)
\psline[linecolor=black, linewidth=0.04](2.8,2.0605078)(2.8,1.5705078)
\end{pspicture}
}
\caption{Two pseudovariety posets}
\end{figure}

Some of our arguments depend on checking certain properties of
specific finite semigroups. Although in principle such calculations
could be carried out by hand they may be done much faster with the
help of a computer. For this purpose, we used the well established
programming language GAP \cite{Gap}. The relevant programs are
included in an appendix.


\section{Preliminaries}\label{sec:prelims}

For standard notation and terminology relating to finite semigroups, refer to
\cite{Alm, Cli}.
We denote by $\mathcal{B}_n(G)$ an $n\times n$ Brandt semigroup over a
group $G$. Note that $\mathcal{B}_n(G)$ is an inverse completely $0$-simple
semigroup.
For elements $s_1,\ldots,s_n$ of a semigroup $S$, we denote $\langle
s_1\ldots,s_n\rangle$ the subsemigroup that they generate.

Jespers and Okni{\'n}ski proved that a completely $0$-simple semigroup $S$ with a maximal subgroup $G$ is
nilpotent if and only if $G$ is nilpotent and $S$ is a Brandt semigroup over $G$ \cite[Lemma 2.1]{Jes-Okn}. 
The next lemma gives a necessary and sufficient condition for a finite semigroup not to be nilpotent.

\begin{lem}[\cite{Jes-Sha}]
  \label{finite-nilpotent}
  A finite semigroup $S$ is not nilpotent if and only if there exist a
  positive integer $m$, distinct elements $x, y\in S$, and elements
  $w_{1},\ldots, w_{m}\in S^{1}$ such that the equalities $x =
  \lambda_{m}(x, y, w_{1}, \ldots, w_{m})$ and $y = \rho_{m}(x,y,
  w_{1}, \ldots, w_{m})$ hold.
\end{lem}

Assume that a finite semigroup $S$ has a proper ideal $M=\mathcal{B}_n(G)$ with $n>1$.
There is an action $\Gamma$ of $S$ on the ${\mathcal R}$-classes of $M$ which plays a key role in~\cite{Jes-Sha3}. In this paper, we consider the dual definition of this action and denote it again by $\Gamma$. We define the action $\Gamma$
of $S$ on the ${\mathcal L}$-classes of $M$, which is a representation (a semigroup homomorphism)
$\Gamma : S\longrightarrow \mathcal{T}$,
where $\mathcal{T}$ denotes the full transformation semigroup on the set $\{1, \ldots, n\} \cup \{\theta\}$.
The definition is as follows, for $1\leq j\leq n$ and $s\in S$,
   $$\Gamma(s)(j) = \left\{ \begin{array}{ll}
      j' & \mbox{if} ~(g;i,j)s=(g';i,j')  ~ \mbox{for some} ~ g, g' \in G, \, 1\leq i \leq n\\
       \theta & \mbox{otherwise}\end{array} \right.$$
and
      $\Gamma (s)(\theta ) =\theta$.
The representation $\Gamma$ is called the \emph{\C} of $S$.

For every $s \in S$, $\Gamma(s)$ may be written as a product of \emph{orbits} which are cycles of the form
$(j_{1}, j_{2}, \ldots, j_{k})$  or sequences of the form $(j_{1}, j_{2}, \ldots, j_{k}, \theta)$, where $1\leq j_{1}, \ldots, j_{k}\leq n$.
The latter orbit means that $\Gamma(s)(j_{i})=j_{i+1}$ for $1\leq i\leq  k-1$, $\Gamma(s)(j_{k})=\theta$, $\Gamma(s)(\theta) =\theta$, and
there does not exist $1\leq r \leq n$ such that  $\Gamma(s)(r)=j_{1}$.
Orbits of the form $(j)$ with $j\in\{1,\ldots,n\}$ are written
explicitly in the orbit decomposition of $\Gamma(s)$. 
By convention, we omit orbits of the form $(j, \theta)$ in
the decomposition of  $\Gamma(s)$ (this is the reason for writing cycles of length one).  If $\Gamma(s)(j)=\theta$ for every $1\leq j \leq n$, then we simply denote $\Gamma(s)$ by $\overline{\theta}$.

If all orbits of
a transformation $\varepsilon$ appear in the expression of $\Gamma(s)$ as a product of disjoint
orbits, then we denote this by $\varepsilon\subseteq \Gamma(s)$.
If 
$$\Gamma(s)(j_{i,1})=j_{i,2}, \ldots,\ \Gamma(s)(j_{i,p_i-1})=j_{i,p_i}\quad (i=1,\ldots,q)$$
then we write  
\begin{equation}\label{[]}
[j_{1,1},j_{1,2}, \ldots,j_{1,p_1};\ldots;j_{i,1},j_{i,2}, \ldots,j_{i,p_i};\ldots; j_{q,1},j_{q,2}, \ldots,j_{q,p_q}] \sqsubseteq \Gamma(s).
\end{equation}
Note that (\ref{[]}) does not imply that $\Gamma(s)(j_{i,p_i})=j_{i,1}$, in contrast with inclusion 
$(j_{i,1},j_{i,2}, \ldots,j_{i,p_i})\subseteq \Gamma(s)$. Also (\ref{[]}) does not imply that there does not exist 
$1\leq r \leq n$ such that $\Gamma(s)(r)=j_{i,1}$, in contrast with the inclusion 
$(j_{i,1},j_{i,2}, \ldots,j_{i,p_i},\theta)\subseteq \Gamma(s)$.

Note that, if $g\in G$ and $1 \leq n_1, n_2 \leq n$ with
$n_1 \neq n_2$ then
$$\Gamma((g;n_1,n_2)) = (n_1,n_2,\theta) \mbox{  and } \Gamma((g;n_1,n_1)) = (n_1).$$
Therefore, if the group $G$ is trivial, then we may view the elements of $M$ as transformations.


Let $T$ be a semigroup with a zero $\theta_T$ and let $M=\mathcal{B}_n(G)$ be a Brandt semigroup.
Let $\Delta$ be a representation of $T$ in the full transformation semigroup on the set $\{ 1, \ldots, n\} \cup \{\theta\}$ such that for
every $t\in T$, $\Delta(t) (\theta ) =\theta$, $\Delta^{-1}(\overline{\theta})= \{\theta_T\}$, and
 $\Delta(t)$ restricted to  $\{ 1, \ldots, n\} \setminus \Delta(t)^{-1}(\theta)$ is injective. 
 The semigroup $S=M \cup^{\Delta} T$ is the $\theta$-disjoint union of
 $M$ and $T$ (that is the disjoint union with the zeros identified).
 The multiplication is such that $T$ and $M$ are subsemigroups,
$$(g;i,j) \, t = \left\{ \begin{array}{ll}
                              (g; i,\Delta (t)(j)) & \mbox{ if } \Delta (t)(j) \neq \theta\\
                              \theta & \mbox{ otherwise,}
                          \end{array} \right.
                          $$
and
$$ t(g;i,j) = \left\{ \begin{array}{ll}
    (g;i',j) & \mbox{  if }  \Delta (t)(i')=i\\
    \theta &\mbox{ otherwise. }
    \end{array} \right.
    $$

Let $\mathsf{V}$ be a pseudovariety of finite semigroups. A
pro-$\mathsf{V}$ semigroup is a compact semigroup that is residually
$\mathsf{V}$. For the pseudovariety $\mathsf{S}$ of all finite
semigroups, we call pro-$\mathsf{S}$ semigroups profinite semigroups.
We denote by $\overline{\Omega}_{A}\mathsf{V}$ the free
pro-$\mathsf{V}$ semigroup on the set $A$. Such free objects are
characterized by appropriate universal properties: the profinite
semigroup $\overline{\Omega}_{A}\mathsf{V}$ comes endowed with a
mapping $\iota\colon A \rightarrow \overline{\Omega}_{A}\mathsf{V}$
such that, for every mapping $\phi \colon A \rightarrow S$ into a
pro-$\mathsf{V}$ semigroup $S$, there exists a unique continuous
homomorphism $\widehat{\phi}\colon \overline{\Omega}_{A}\mathsf{V}
\rightarrow S$ such that $\widehat{\phi}\circ\iota=\phi$. For more
details on this topic we refer the reader to \cite{Alm}.

Let $r$ be an integer. We denote the free pro-$\mathsf{V}$ semigroup on the set $\{x_1,\ldots,x_r\}$ by $\overline{\Omega}_{r}\mathsf{V}$.
Recall that a pseudoidentity (over $\mathsf{V}$) is a formal equality $\pi = \rho$ between  $\pi,\rho\in  \overline{\Omega}_{r}\mathsf{V}$ for some integer $r$.
For a set $\Sigma$ of $\mathsf{V}$-pseudoidentities, we denote by $\llbracket \Sigma \rrbracket_{\mathsf{V}}$ (or simply $\llbracket \Sigma \rrbracket$ if $\mathsf{V}$ is understood from the context) the class of all $S \in \mathsf{V}$ that satisfy all pseudoidentities from $\Sigma$. Reiterman \cite{Rei} proved that a subclass $\mathsf{V}$ of a pseudovariety $\mathsf{W}$ is a pseudovariety if and only if $\mathsf{V}$ is of the form $\llbracket \Sigma \rrbracket_{\mathsf{W}}$ for some set $\Sigma$ of $\mathsf{W}$-pseudoidentities. 

Let $1\leq r$, $S$ be a pro-\pv V semigroup, and
$\pi\in\overline{\Omega}_{r}\mathsf{V}$. The operation $\pi_S\colon
S^r \rightarrow S$ is defined as follows:
$$\pi_S(s_1,\ldots,s_r)=\widehat{\phi}(\pi)$$ where $\phi\colon
\{x_1,\ldots,x_r\} \rightarrow S$ is the mapping given by
$\phi(x_i)=s_{i}$, for all $1\leq i\leq r$ and $\widehat{\phi}$ was
defined above. The families of operations $(\pi_S)_{S\in\pv V}$ that
may be obtained in this way are precisely those that commute with
homomorphisms between members of~\pv V and are called implicit
operations. Moreover, the correspondence
$\pi\in\overline{\Omega}_{r}\mathsf{V}\mapsto(\pi_S)_{S\in\pv V}$ is
injective and one often identifies $\pi$ with its image.

For an element $s$ of a finite semigroup $S$, $s^{\omega}$ denotes the
unique idempotent power of $s$. This defines a unary implicit
operation $x \mapsto x^{\omega}$ on finite semigroups. In particular,
given an element $s$ of a profinite semigroup $S$, we may consider
$s^\omega=(x^\omega)_S(s)$. Note that
$s^{\omega}$ is the limit of the sequence $(s^{n!})_n$.

For a profinite semigroup $S$, let $\mathrm{End}(S)$ be the monoid of
all continuous endomorphisms of $S$. In case $S$ has a finitely
generated dense subsemigroup, it turns out that $\mathrm{End}(S)$ is a
profinite monoid under the pointwise convergence topology, that is, as
a subspace of the product space $S^S$ \cite{Hunter:1983} (see also
\cite{Steinberg:2010pm}). In particular, given $\varphi$ in
$\mathrm{End}(S)$, we may consider the continuous endomorphism
$\varphi^\omega$, which is the pointwise limit of the sequence
$\lim\varphi^{n!}$.

As examples, consider the pseudovarieties $\mathsf{G_{nil}}$ and
$\mathsf{BG}$. We have
\begin{displaymath}
  \mathsf{G_{nil}}
  =\llbracket
  \psi^{\omega}(x)=x^{\omega},x^{\omega}y=yx^{\omega}=y
  \rrbracket 
\end{displaymath}
where $\psi$ is the continuous endomorphism of the free profinite
semigroup on $\{x,y\}$ such that $\psi(x)=x^{\omega-1}y^{\omega-1}xy$ and
$\psi(y)=y$ \cite[Example 4.15(2)]{Alm2}, and
\begin{displaymath}
  \mathsf{BG}=\llbracket (ef)^{\omega}=(fe)^{\omega} \rrbracket
\end{displaymath}
where $e=x^{\omega}$, $f=y^{\omega}$ (see, for example,
\cite[Exercise~5.2.7]{Alm}).



\section{The pseudovariety \pv{MN}}
\label{sec:pvy-pvmn}

Jespers and Riley \cite{Jes-Ril} called a semigroup $S$ \emph{weakly
  Mal’cev nilpotent} (WMN) if 
\begin{align*}
  &\exists n\ge0\ \forall a,b\in S\ \forall c_{1},c_2\in S^{1},\\
  &\quad
    \lambda_{n}(a,b,c_{1},c_2,c_1,c_2,\ldots)
    = \rho_{n}(a,b,c_{1},c_2,c_1,c_2,\ldots).
\end{align*}
They proved that a linear semigroup $S$ is MN if and only
if $S$ is WMN \cite[Corollary
12]{Jes-Ril}. 
Recall that a linear semigroup is a subsemigroup of a matrix monoid
$M_n(K)$ over a field $K$. In particular, all finite semigroups are
linear. The following lemma gives a slightly simpler characterization
of Mal'cev nilpotency for finite semigroups by avoiding referring to
the identity element.

\begin{lem}\label{finite-nil-dif}
  A finite semigroup $S$ is Mal'cev nilpotent if and only if
  \begin{align}
    \label{eq:finite-nil-dif}
    &\exists n\ge0\ \forall a,b,c_1,c_2\in S,\\
    &\quad\notag
      \lambda_n(a,b,c_1,c_2,c_1,c_2,\ldots) =
      \rho_n(a,b,c_1,c_2,c_1,c_2,\ldots).
  \end{align}
\end{lem}

\begin{proof}
  Suppose that there exists a finite semigroup $S$ that satisfies
  \eqref{eq:finite-nil-dif}
  but is not Mal'cev nilpotent.

  If $S\not\in \mathsf{BG}_{nil}$, then there exists a regular
  ${\mathcal J}$-class $\mathcal{M}^0(G,n,m;P)\setminus \{\theta\}$ of
  $S$ such that one of the following conditions holds:
  \begin{enumerate}[label=(\roman*)]
  \item\label{item:finite-nil-dif-1} $G$ is not a nilpotent group;
  \item\label{item:finite-nil-dif-2} there exist integers $1\leq
    i_1<i_2\leq n$ and $1\leq j\leq m$ such that $p_{ji_1},p_{ji_2}\neq
    \theta$;
  \item\label{item:finite-nil-dif-3} there exist integers $1\leq i\leq
    n$ and $1\leq j_1<j_2\leq m$ such that $p_{j_1i},p_{j_2i}\neq
    \theta$.
  \end{enumerate}
  In case \ref{item:finite-nil-dif-1}, $G$ is not MN and, thus, $G$ is
  not WMN. Since $G$ has an identity, $G$ does not satisfy the
  condition of the lemma, a contradiction. If
  \ref{item:finite-nil-dif-2} holds, then
  \begin{align*}
    &\lambda_n((1_G;i_1,j),(1_G;i_2,j),
      (1_G;i_1,j),(1_G;i_1,j),(1_G;i_1,j),(1_G;i_1,j),\ldots) \neq\\
    &\rho_n((1_G;i_1,j),(1_G;i_2,j),
      (1_G;i_1,j),(1_G;i_1,j),(1_G;i_1,j),(1_G;i_1,j),\ldots),
  \end{align*}
  for every integer $n\geq 0$, which contradicts the assumption that
  $S$ satisfies~\eqref{eq:finite-nil-dif}. Similarly, we reach a
  contradiction assuming Condition \ref{item:finite-nil-dif-3}.

  Now, suppose that $S\in \mathsf{BG}_{nil}$. Since
  $S\not\in\mathsf{MN}$, by Lemma~\ref{finite-nilpotent} there exist a
  positive integer $m$, distinct elements $x, y\in S$ and elements $
  w_{1}, w_{2}, \ldots, w_m\in S^{1}$ such that
  \begin{eqnarray}
    \label{nil-fi}
    x = \lambda_{m}(x, y, w_1, \ldots, w_m)
    \mbox{ and }
    y = \rho_{m}(x,y, w_1,\ldots, w_m).
  \end{eqnarray}
  As $x\neq y$ and $S\in \mathsf{BG}_{nil}$, by \eqref{nil-fi}, there
  exists a regular ${\mathcal
    J}$-class $$M=\mathcal{M}^0(G,n,n;I_n)\setminus \{\theta\}$$
  of $S$ such that $x,y\in M$. Then, there exist elements
  $(g;i,j),(g';i',j')\in M$ such that $x=(g;i,j)$ and $y=(g';i',j')$.
  By \eqref{nil-fi}, we also obtain that $[j,i';j',i] \sqsubseteq
  \Gamma(w_1)$ and $[j,i;j',i'] \sqsubseteq \Gamma(w_2)$, where
  $\Gamma$ is an ${\mathcal L}$-representation of $S$. If $w_1=1$,
  then we have $j=i'$ and $j'=i$. In case $i\neq j$, we get
  $(i,j)\subseteq \Gamma(w_2)$. Thus, we have
  \begin{align*}
    &\lambda_{n}((1_G;i,i),(1_G;j,j),w_2,w_2^2,w_2,w_2^2,\ldots) \neq\\
    &\rho_{n}((1_G;i,i),(1_G;j,j),w_2,w_2^2,w_2,w_2^2,\ldots),
  \end{align*}
  for every integer $n\ge0$, which contradicts the assumption that $S$
  satisfies~\eqref{eq:finite-nil-dif}. If $i=j$, then we have
  $i=j=i'=j'$ which entails that $G$ is not nilpotent, in
  contradiction with the assumption that $S\in \mathsf{BG}_{nil}$.
  Also, if $w_2=1$, we reach a similar contradiction. Therefore, we
  have $w_1,w_2\neq 1$. If $(i,j)=(i',j')$, then we conclude that $G$
  is not nilpotent, contradicting the assumption that $S\in
  \mathsf{BG}_{nil}$. Hence, we have $(i,j)\neq(i',j')$. Now, as
    $[j,i';j',i] \sqsubseteq \Gamma(w_1)$ and $[j,i;j',i'] \sqsubseteq
    \Gamma(w_2)$ we have
  \begin{align*}
    &\lambda_{n}((1_G;i,j),(1_G;i',j'),w_1,w_2,w_1,w_2,\ldots) \neq\\
    &\rho_{n}((1_G;i,j),(1_G;i',j'),w_1,w_2,w_1,w_2,\ldots),
  \end{align*} for every integer $n\ge0$. 
  
  The result follows.
\end{proof}

Combining Lemma~\ref{finite-nil-dif} with the same method as in the
proof of \cite[Lemma 2.2]{Jes-Sha}, one obtains the following lemma.

\begin{lem}
  \label{finite-nilpotentW1W2}
  A finite semigroup $S$ is not nilpotent if and only if there exist a
  positive integer $m$, distinct elements $x, y\in S$ and elements $w_1,w_2\in S$
  such that the equalities 
  $$x =
  \lambda_{m}(x, y, w_1, w_2,w_1, w_2,\ldots)\ \text{and}\ y = \rho_{m}(x,y,w_1, w_2,w_1, w_2,\ldots)$$ hold.
\end{lem}

Also using Lemma~\ref{finite-nil-dif}, one immediately deduces the
following theorem.

\begin{thm}
  \label{PS MN}
  The following equality holds
  \begin{displaymath}
    \mathsf{MN}=\llbracket \phi^{\omega}(x)=\phi^{\omega}(y) \rrbracket
  \end{displaymath}
  where $\phi$ is the continuous endomorphism of the free profinite
  semigroup on $\{x,y,z,t\}$ such that $\phi(x)=xzytyzx$,
  $\phi(y)=yzxtxzy$, $\phi(z)=z$, and $\phi(t)=t$.
\end{thm}



In particular, Theorem~\ref{PS MN} shows that the pseudovariety
$\mathsf{MN}$ is finitely based.



\section{The pseudovariety \pv{PE}}
\label{sec:pvy-pvpe}

As in \cite{Jes-Ril}, we denote by $F_{7}$ the 7-element semigroup
which is the $\theta$-disjoint union of the Brandt semigroup
$\mathcal{B}_2(\{1\})$ with the cyclic group $\{1,u\}$ of order~$2$
with a zero $\theta$ adjoined:
\begin{eqnarray} \label{ex2}
F_{7} &=& \mathcal{B}_2(\{1\}) \cup^{\Gamma_{F_7}} \{1,u,\theta\}
,\end{eqnarray}
where $\Gamma_{F_7}(u)=(1,2)$ and $\Gamma_{F_7}(1)=(1)(2)$. Note that
$F_{7}=\langle u,(1;1,1)\rangle$.
In fact, $F_{7}$ is a member of the family of minimal non-cryptic inverse semigroups considered by Reilly in \cite{Reil}.

We recall that a semigroup $S$ divides a semigroup $T$ and we write $S\prec T$ if $S$
is a homomorphic image of a subsemigroup of $T$. 
We say that a finite semigroup $S$ is $\times$-prime whenever, for all finite semigroups $T_1$ and $T_2$, if $S\prec T_1\times T_2$, then $S\prec T_1$ or $S\prec T_2$. For example, a right or left zero semigroup with 2 elements is $\times$-prime \cite[Exercise 9.3.1(a)]{Alm}.
A more sophisticated example of $\times$-prime semigroups is $F_{7}$
\cite[Theorem 7.3.10]{Rho-Ste}. It follows that the 
class of all finite semigroups $S$ such that $F_7\not\prec S$ is a
pseudovariety.

In \cite{Jes-Ril}, it is proved that a finite semigroup $S$ is positively Engel if and only if the following properties hold:
\begin{enumerate}
\item all non-null  principal factors of $S$ are inverse semigroups whose maximal subgroups are nilpotent groups.
\item $S$ does not have an epimorphic image which admits $F_{7}$ as a subsemigroup.
\end{enumerate}
We rewrite that result as Theorem~\ref{PE F_7}.

\begin{thm}
  \label{PE F_7}
  Let $S$ be a finite semigroup. The semigroup $S$ is positively Engel
  if and only if $S\in\mathsf{BG_{nil}}$ and $F_7\not\prec S$.
\end{thm}

\begin{proof}
  Suppose that $S\in\mathsf{BG_{nil}}$ and $S\not\in\mathsf{PE}$. Then
  there exists an onto homomorphism $\phi:S\rightarrow S'$ such that
  $F_7$ is a subsemigroup of $S'$. Let $U=\phi^{-1}(F_7)$. Then
  $\phi\mid_U:U\rightarrow F_7$ is an onto homomorphism, so that
  $F_7\prec S$.

Now, suppose that $F_7\prec S$. Then, there exist a subsemigroup $U$
of $S$ and an onto homomorphism $\phi:U\rightarrow F_7$. There exist elements $a,w\in U$ such that $\phi(a)=(1;1,1)$ and
$\phi(w)=u$. Since 
\begin{align*}
\lambda_{2n}((1;1,1),u,1,1,u,u^2,\ldots,u^{n-2})&=(1;1,1),\\
\rho_{2n}((1;1,1),u,1,1,u,u^2,\ldots,u^{n-2})&=(1;2,2),\\
\lambda_{2n-1}((1;1,1),u,1,1,u,u^2,\ldots,u^{n-2})&=(1;1,2),\\
\rho_{2n-1}((1;1,1),u,1,1,u,u^2,\ldots,u^{n-2})&=(1;2,1),
\end{align*} 
for every positive integer $n$, we have
$$\lambda_n(a,w,1,1,w,w^2,\ldots,w^{n-2})\neq\rho_n(a,w,1,1,w,w^2,\ldots,w^{n-2}),$$
for every positive integer $n$. This shows that $S\not\in\mathsf{PE}$.
The result follows.
\end{proof}

\begin{prop}
  \label{PS PE identity}
  The sequence
  \begin{displaymath}
    \left\{\bigl(
      \lambda_{n!}(x,y,z,z^{2},\ldots ,z^{n!}),
      \rho_{n!}(x,y,z,z^{2},\ldots,z^{n!})\bigr)
      \mid n\in \mathbb{N}\right\}
  \end{displaymath}
  converges in
  $(\overline{\Omega}_{\{x,y,z\}}\mathsf{S})^2$.
\end{prop}

\begin{proof}
  For simplicity, we
  let
  \begin{displaymath}
    \overline{\lambda}_{i}=\lambda_{i}(x,y,z,\ldots ,z^{i})
    \mbox{ and }
    \overline{\rho}_{i}=\rho_{i}(x,y,z,\ldots ,z^{i})
  \end{displaymath}
  for every $1\leq i$. Suppose that $S$ is a finite semigroup and
  $\phi:\overline{\Omega}_{\{x,y,z\}}\mathsf{S}\rightarrow S$ is a
  continuous homomorphism. Let $k=\abs{S}$. By the pigeonhole principle, there
  exist positive integers $t$ and $r$ such that $t<r\leq k^3+1$ and
  \begin{align*}
    &\Bigl(\phi\bigl(\lambda_{t}(x,y,z,\ldots ,z^{t})\bigr),
      \phi\bigl(\rho_{t}(x,y,z,\ldots ,z^{t})\bigr),\phi(z^{t+1})\Bigr)\\
    &=\Bigl(\phi\bigl(\lambda_{r}(x,y,z,\ldots ,z^{r})\bigr),
      \phi\bigl(\rho_{r}(x,y,z,\ldots ,z^{r})\bigr),\phi(z^{r+1})\Bigr).
  \end{align*}
  Let $s$ be a positive integer such that $t\leq s(r-t)< r$. Then, the
  equality
  \begin{displaymath}
    \bigl(\phi(\overline{\lambda}_{n!}),\phi(\overline{\rho}_{n!})\bigr)
    =\bigl(\phi(\overline{\lambda}_{(n+1)!}),\phi(\overline{\rho}_{(n+1)!})\bigr)
  \end{displaymath}
  holds for every $n>s(r-t)$. The result follows.
\end{proof}

We denote by $(\lambda_{\mathsf{PE}},\rho_{\mathsf{PE}})$ the limit of the sequence $(\overline{\lambda}_{n!},\overline{\rho}_{n!})_{n}$  in the preceding proposition.

\begin{thm} \label{PS PE}
The following equality holds
\begin{center}$\mathsf{PE}=\mathsf{TM}\cap\llbracket \phi(\lambda_{\mathsf{PE}})=\phi(\rho_{\mathsf{PE}})\rrbracket$\end{center} 
where $\phi$ is the continuous endomorphism of the free profinite semigroup on $\{x,y,z\}$ such that $ \phi(x)=xzzx$, $\phi(y)=zxxz$, 
$\phi(z)=z$.
\end{thm}

\begin{proof}
  Suppose that $S\in \mathsf{TM}\cap\llbracket
  \phi(\lambda_{\mathsf{PE}})=\phi(\rho_{\mathsf{PE}})\rrbracket$ and
  $F_7\prec S$. Then there exist a subsemigroup $U$ of $S$ and an onto
  homomorphism $\psi:U\rightarrow F_7$. Let $a,b$ be elements of $U$
  such that $\psi(a)=(1;1,1)$ and $\psi(b)=u$. Then, we have
  $\lambda_n(a,b,1,1,b,b^2,\ldots,b^{n-2})\neq\rho_n(a,b,1,1,b,b^2,\ldots,b^{n-2})$
  for every positive integer $n$, which contradicts the assumption
  that $S$ satisfies the pseudoidentity
  $\phi(\lambda_{\mathsf{PE}})=\phi(\rho_{\mathsf{PE}})$. Therefore,
  by Theorem~\ref{PE F_7}, $S\in \mathsf{PE}$.

  The converse follows at once from the definition of the
  pseudovariety $\mathsf{PE}$.
\end{proof}

\section{The pseudovariety $\pv{NMN}_{4}$}
\label{sec:pvy-pvNMN_4}

We denote by $F_{12}$ the 12-element subsemigroup of the full
transformation semigroup on the set $\{1,2,3\} \cup \{\theta\}$ given
by the union of the Brandt semigroup $\mathcal{B}_3(\{1\})$ with the
semigroup $\langle w_1,w_2\rangle$,
\begin{equation*}
  F_{12} = \mathcal{B}_3(\{1\}) \cup \langle w_1,w_2\rangle,
\end{equation*}
where $w_1=(1)(3,2,\theta)$ and $w_2=(3,1,2,\theta)$. By \cite[Lemma
3.2]{Jes-Sha3}, the semigroup $F_{12}$ is not nilpotent but the
subsemigroup $\langle w_1,w_2 \rangle$ is nilpotent.

Let $S$ be a semigroup in which $\mathcal{B}_n(G)$ is a proper ideal.
We define the functions $\epsilon_1$ and $\epsilon_2$ from $S$ to the
power set $\mathcal{P}(\{1,\ldots,n\})$ of the set $\{1,\ldots,n\}$
as follows:
\begin{align*}
  \epsilon_1(w)
  &=\{i\mid 1\leq i\leq n\mbox{ and }\Gamma(w)(i)\neq \theta\},\\
  \epsilon_2(w)
  &=\{i\mid 1\leq i\leq n
    \mbox{ and there exists an integer $j$ such that } \Gamma(w)(j)=i\}
\end{align*}
for every $w\in S$. We also define the functions $\iota_1$ and
$\iota_2$ from $S$ to $\mathbb{N}$ as follows:
\begin{displaymath}
  \iota_1(w)=\vert\epsilon_1(w)\vert \mbox{ and }\iota_2(w)=\vert\epsilon_2(w)\vert
\end{displaymath}
for every $w\in S$.

Let $\mathsf{Excl}(F_7,F_{12})$ denote the class of all finite
semigroups $S$ that are divisible neither by $F_7$ nor by $F_{12}$. We
say that the semigroup $S$ is \emph{$\M$} if $S\in\mathsf{BG_{nil}}$
and $S\in \mathsf{Excl}(F_7,F_{12})$. We proceed with several
technical lemmas that serve to show that the class $\mathsf{\M}$ of
all finite $\M$ semigroups is a pseudovariety. In view of the
definition and Theorem~\ref{PE F_7}, $\mathsf{\M}$ is contained in the
pseudovariety $\mathsf{PE}$. It amounts to a routine calculation (see
the appendix) to verify that the semigroup $F_{12}$ satisfies the
identity $\lambda_{3}(a,b,1,1,c) =\rho_{3} (a,b,1,1,c)$, for $a,b\in
F_{12}$ and $c\in F_{12}^{1}$. Hence, we have $F_{12} \in
\mathsf{PE}$. Therefore, as $F_{12}$ is not $\M$, the pseudovariety $\mathsf{\M}$
is strictly contained in $\mathsf{PE}$.
 
\begin{lem}
  \label{iota}
  Let $S$ be a semigroup admitting $\mathcal{B}_n(G)$ as a proper
  ideal and let $w$ be an element of $S$. Suppose
  that
  \begin{displaymath}
    (a_1,\ldots,a_m,\theta)\subseteq\Gamma(w)
  \end{displaymath}
  for some integers $1\leq a_1,\ldots,a_m\leq n$ with $m>1$. Then, we
  have
  \begin{displaymath}
    \max\{\iota_1(w^py),\iota_1(yw^p),\iota_2(w^py),\iota_2(yw^p)\}
    <\min\{\iota_1(w),\iota_2(w)\}
  \end{displaymath}
  for every element $y\in S^1$ and integer $p>1$.
\end{lem}

\begin{proof}
  Since $\epsilon_1(w^py)\subseteq\epsilon_1(w)$ and
  $a_{m-1}\not\in\epsilon_1(w^py)$, we obtain
  $\iota_1(w^py)<\iota_1(w)$.

  There is a one-to-one function $\phi$ from $\epsilon_1(yw^p)$ to
  $\epsilon_1(w)$ as follows: for elements $a\in \epsilon_1(yw^p)$ and
  $b\in\epsilon_1(w)$, let $\phi(a)=b$ if $\Gamma(yw^{p-1})(a)=b$. We
  deduce that $\iota_1(yw^p)\leq\iota_1(w)$. Since $m>1$, we have
  $a_{1}\in \epsilon_1(w)$. Hence, we obtain
  $\iota_1(yw^p)<\iota_1(w)$.

  There is a one-to-one function $\phi$ from $\epsilon_2(w^py)$ to
  $\epsilon_1(w)$ as follows: for elements $a\in \epsilon_2(w^py)$ and
  $b\in\epsilon_1(w)$, let $\phi(a)=b$ if $\Gamma(wy)(b)=a$. We deduce
  that $\iota_2(w^py)\leq\iota_1(w)$. Since $a_{1}\in \epsilon_1(w)$,
  we obtain $\iota_2(w^py)<\iota_1(w)$.

  There is a one-to-one function $\phi$ from $\epsilon_2(yw^p)$ to
  $\epsilon_1(w)$ as follows: for elements $a\in \epsilon_2(yw^p)$ and
  $b\in\epsilon_1(w)$, let $\phi(a)=b$ if $\Gamma(w)(b)=a$. We deduce
  that $\iota_2(yw^p)\leq\iota_1(w)$. Since $a_1\in \epsilon_1(w)$, we
  obtain $\iota_2(yw^p)<\iota_1(w)$.

%
%
%
%

  Similarly, we have
  $\iota_1(w^py),\iota_1(yw^p),\iota_2(w^py),\iota_2(yw^p)<\iota_2(w)$.
\end{proof}

\begin{lem}
  \label{m,l}
  Let $S$ be a finite semigroup. Suppose that there exist a proper
  ideal $S'$ and an element $w$ of $S$ that satisfy the following
  properties:
  \begin{enumerate}[label=(\roman*)]
  \item the semigroup $M = \mathcal{B}_n(G)$ is a proper ideal of
    $S/S'$ for some finite group $G$ and integer $n\geq 2$;
  \item $(m,l) \subseteq \Gamma(w)$ for some distinct positive
    integers $1\leq l,m\leq n$ where $\Gamma$ is an \C\ of $S/S'$.
  \end{enumerate}
  Then, we must have $F_{7}\prec S$.
\end{lem}

\begin{proof}
  We claim that there is an onto homomorphism from the subsemigroup
  $Z=\langle (1;l,l),w \rangle$ of $S$ to $F_{7}$. Define the function
  $\psi:Z\rightarrow F_7$ as follows:
  \begin{itemize}
  \item if $z\in S'$ or $\Gamma(z)(l)=\Gamma(z)(m)=\theta$ then $\psi(z)=\theta$;
  \item if $(m,l) \subseteq \Gamma(z)$ then $\psi(z)=u$;
  \item if $(m)(l)\subseteq \Gamma(z)$ then $\psi(z)=1$; 
  \item if $(l) \subseteq \Gamma(z)$ and $\Gamma(z)(m)=\theta$ then $\psi(z)=(1;1,1)$; 
  \item if $(l,m,\theta) \subseteq \Gamma(z)$ then $\psi(z)=(1;1,2)$; 
  \item if $(m,l,\theta) \subseteq \Gamma(z)$ then $\psi(z)=(1;2,1)$;
  \item if $(m) \subseteq \Gamma(z)$ and $\Gamma(z)(l)=\theta$ then $\psi(z)=(1;2,2)$.
  \end{itemize}
  Since $\Gamma$ is an \C\ $S/S'$, $\psi$ is a homomorphism.
\end{proof}

\begin{lem}
  \label{m,l,m'}
  Let $S$ be a finite semigroup. Suppose that there exist a proper
  ideal $S'$ and elements $w_1$ and $w_2$ of $S$ that satisfy the
  following properties:
  \begin{enumerate}[label=(\roman*)]
  \item the semigroup $M = \mathcal{B}_n(G)$ is a proper ideal of
    $S/S'$ for some finite group $G$ and integer $n\geq 3$;
  \item the relations $[m,l,m'] \sqsubseteq \Gamma(w_1)$, $[m, m']
    \sqsubseteq \Gamma(w_2)$, and $(l)\subseteq \Gamma(w_2)$ hold for
    some pairwise distinct positive integers $1\leq l,m,m'\leq n$
    where $\Gamma$ is an \C\ of $S/S'$.
  \end{enumerate}
  Then, at least one of the conditions $F_{7}\prec S$ or $F_{12}\prec
  S$ holds.
\end{lem}

\begin{proof}
  Suppose that the integer $m'$ is in a cycle of $\Gamma(w_1)$ with
  length $L_{w_1}$. Then $(l,m')\subseteq \Gamma(w_1^{L_{w_1}-1}w_2)$.
  Hence, by Lemma~\ref{m,l}, there is an onto homomorphism from the
  subsemigroup $\langle (1;l,l),w_1^{L_{w_1}-1}w_2 \rangle$ of $S$ to
  $F_{7}$. Suppose next that the integer $m'$ is in a cycle of
  $\Gamma(w_2)$ with length $L_{w_2}$. Then $(l,m)\subseteq
  \Gamma(w_1w_2^{L_{w_2}-1})$. Hence, by Lemma~\ref{m,l}, there is an
  onto homomorphism from the subsemigroup $\langle
  (1;l,l),w_1w_2^{L_{w_2}-1} \rangle$ of $S$ to $F_{7}$.

  Now, suppose that the integer $m'$ is in non-cycle orbits of both
  $\Gamma(w_1)$ and $\Gamma(w_2)$. Then the following set $\Sigma$ is
  not empty:
  \begin{align*}
    \Sigma= \bigl\{
    &(y_1,y_2;a,b,c)\mid y_1,y_2\in S,1\leq a,b,c\leq n,[a,c,b]
      \sqsubseteq \Gamma(y_1), [a, b] \sqsubseteq \Gamma(y_2),\\
    &(c)\subseteq \Gamma(y_2),\mbox{ the orbit that contains $[a,c,b]$
      in $\Gamma(y_1)$  is not a cycle}\\ 
    &\mbox{and the orbit that contains $[a,b]$ in $\Gamma(y_2)$ is not a
      cycle}
      \bigr\}.
  \end{align*}
  Since $S$ is finite, there exists an element
  $\Delta=(y_1,y_2;a,b,c)\in \Sigma$ such that there does not exist an
  element $\Delta'=(u_1,u_2;e,f,d)\in \Sigma$ for which
  $\iota_i(u_k)<\iota_i(y_k)$ for every $i,k\in\{1,2\}$.

  Since $\Delta\in \Sigma$, there exist integers
  \begin{displaymath}
    1\leq
    a_1,\ldots,a_t,a'_1,\ldots,a'_{t'},b_1,\ldots,b_r,b'_1,\ldots,b'_{r'}\leq 
    n
  \end{displaymath}
  such that
  \begin{displaymath}
    (a_1,\ldots,a_t,c,b_1,\ldots,b_r,\theta)
    \subseteq \Gamma(y_1),\
    (c)(a'_1,\ldots,a'_{t'},b'_1,\ldots,b'_{r'},\theta)
    \subseteq \Gamma(y_2),
  \end{displaymath}
  $a_t=a'_{t'}=a$ and $b_1=b'_{1}=b$. Suppose that
  $\{a_1,\ldots,a_t\}\cap\{b'_1,\ldots,b'_{r'}\}\neq\emptyset$. Then,
  there exist integers $1\leq \gamma \leq t$ and $1\leq \lambda \leq
  r'$ such that $a_{\gamma}=b'_{\lambda}$. Thus, we obtain
  $(a_{\gamma})\subseteq \Gamma(y_1^{t-\gamma }y_2^{\lambda})$,
  $[a_{\gamma+1},c] \sqsubseteq \Gamma(y_1^{t-\gamma }y_2^{\lambda})$
  and $[a_{\gamma+1},a_{\gamma},c] \sqsubseteq
  \Gamma(y_1^{t-\gamma+1}y_2^{\lambda-1})$. If the integer $c$ is in a
  cycle of $\Gamma(y_1^{t-\gamma }y_2^{\lambda})$ or
  $\Gamma(y_1^{t-\gamma+1}y_2^{\lambda-1})$, then similarly there is
  an onto homomorphism from a subsemigroup of $S$ to $F_{7}$.
  Otherwise, the integer $c$ is in non-cycle orbits of
  $\Gamma(y_1^{t-\gamma+1}y_2^{\lambda-1})$ and $\Gamma(y_1^{t-\gamma
  }y_2^{\lambda})$. We conclude that
  $(y_1^{t-\gamma+1}y_2^{\lambda-1},y_1^{t-\gamma
  }y_2^{\lambda};a_{\gamma+1},c,a_{\gamma})\in \Sigma$. Since
  $a_{\gamma}\not\in\{a_t,b'_1\}$, we obtain $\gamma<t$ and $1<
  \lambda$. Then, by Lemma~\ref{iota}, we have
  \begin{align*}
    &\iota_1(y_1^{t-\gamma+1}y_2^{\lambda-1}) <\iota_1(y_1),\
      \iota_2(y_1^{t-\gamma+1}y_2^{\lambda-1}) <\iota_2(y_1), \\
    &\iota_1(y_1^{t-\gamma}y_2^{\lambda})<\iota_1(y_2),
      \mbox{ and }
      \iota_2( y_1^{t-\gamma } y_2^{\lambda})<\iota_2(y_2),
  \end{align*}
  a contradiction.

  Suppose next that
  $\{a'_1,\ldots,a'_{t'}\}\cap\{b_1,\ldots,b_r\}\neq\emptyset$. Then,
  there exist integers $1\leq \gamma \leq t'$ and $1\leq \lambda \leq
  r$ such that $a'_{\gamma}=b_{\lambda}$. Thus, we obtain
  $(b_{\lambda})\subseteq \Gamma(y_2^{t'-\gamma+1}y_1^{\lambda-1})$,
  $[c,b_{\lambda-1}] \sqsubseteq
  \Gamma(y_2^{t'-\gamma+1}y_1^{\lambda-1})$ and
  $[c,b_{\lambda},b_{\lambda-1}] \sqsubseteq
  \Gamma(y_2^{t'-\gamma}y_1^{\lambda})$. If the integer $c$ is in a
  cycle of $\Gamma(y_2^{t'-\gamma+1}y_1^{\lambda-1})$ or
  $\Gamma(y_2^{t'-\gamma}y_1^{\lambda})$, then similarly there is an
  onto homomorphism from a subsemigroup of $S$ to $F_{7}$. Otherwise,
  the integer $c$ is in non-cycle orbits of
  $\Gamma(y_2^{t'-\gamma}y_1^{\lambda})$ and
  $\Gamma(y_2^{t'-\gamma+1}y_1^{\lambda-1})$. It follows that
  $(y_2^{t'-\gamma}y_1^{\lambda},y_2^{t'-\gamma+1}y_1^{\lambda-1};
  c,b_{\lambda-1},b_{\lambda})\in \Sigma$. Since
  $a'_{\gamma}\not\in\{a'_{t'},b_1\}$, we obtain $\gamma<t'$ and $1<
  \lambda$. Then, by Lemma~\ref{iota}, we have
  \begin{align*}
    &\iota_1(y_2^{t'-\gamma}y_1^{\lambda})<\iota_1(y_1),\
      \iota_2(y_2^{t'-\gamma}y_1^{\lambda})<\iota_2(y_1),\\
    &\iota_1(y_2^{t'-\gamma+1}y_1^{\lambda-1})<\iota_1(y_2),
      \mbox{ and }
      \iota_2(y_2^{t'-\gamma+1}y_1^{\lambda-1})<\iota_2(y_2),
  \end{align*}
  which is again a contradiction.

  We may, therefore, assume that
  \begin{displaymath}
    \{a_1,\ldots,a_t\}\cap\{b'_1,\ldots,b'_{r'}\}=\emptyset
    \mbox{ and }
    \{a'_1,\ldots,a'_{t'}\}\cap\{b_1,\ldots,b_r\}=\emptyset.
  \end{displaymath}
  We denote by $F'_{12}$ the subsemigroup of the full transformation
  semigroup on the set $\{a_t,b_1,c\} \cup \{\theta'\}$ given by the
  union of the Brandt semigroup $\mathcal{B}_3(\{1\})$ with the
  semigroup $\langle w_1,w_2\rangle$,
  \begin{eqnarray}
    F'_{12} &=& \mathcal{B}_3(\{1\}) \cup^{\Gamma_{F'_{12}}}\langle w_1,w_2\rangle,
  \end{eqnarray}
  where $\Gamma_{F'_{12}}(w_1)=(c)(a_t,b_1,\theta')$ and
  $\Gamma_{F'_{12}}(w_2)=(a_t,c,b_1,\theta')$ and $\Gamma_{F'_{12}}$
  is an \C\ of $F'_{12}$. The semigroup $F'_{12}$ is generated by the
  elements $(1;a_t,b_1)$, $w_1$ and $w_2$. On the other hand, the
  semigroup $F_{12}$ is also generated by three elements, namely
  $(1;2,3)$, $(1)(2,3,\theta)$ and $(3,1,2,\theta)$ which are
  basically the same as for $F'_{12}$ up to renaming the elements of
  the set on which they act: $a_t\mapsto 2$, $b_1\mapsto 3$, and
  $c\mapsto 1$. Hence, the transformation semigroups $F'_{12}$ and
  $F_{12}$ are isomorphic. We claim that there is an onto homomorphism
  from the subsemigroup $Z=\langle y_1,y_2,(1;b_1,a_t)\rangle$ onto
  $F'_{12}$. Let $\psi:Z\rightarrow F'_{12}$ be the function defined
  as follows:
  \begin{enumerate}
  \item if $z\in S'$, then $\psi(z)=\theta'$;
  \item otherwise, $\Gamma(z)(i)=j$ if and only if
    $\Gamma'(\psi(z))(i)=j$ for every $i,j\in\{a_t,b_1,c\}$.
  \end{enumerate}
  We prove that $\psi$ is a homomorphism. Let $s_1,s_2\in Z$ and
  $i\in\{a_t,b_1,c\}$. Suppose that
  $\Gamma'(\psi(s_1))(\Gamma'(\psi(s_2))(i))=\alpha$ for some
  $\alpha\in\{a_t,b_1,c\}$. Then, we obtain
  $\Gamma'(\psi(s_1))(i)=\beta$ for some $\beta\in\{a_t,b_1,c\}$ and,
  thus, $\Gamma(s_1)(i)=\beta$ and $\Gamma(s_2)(\beta)=\alpha$, whence
  $\Gamma(s_1s_2)(i)=\alpha$. Now, suppose
  that $$\Gamma'(\psi(s_1))(\Gamma'(\psi(s_2)(i)))=\theta'.$$
  If $\Gamma'(\psi(s_1s_2))(i)\neq \theta'$, then there exists an
  integer $\gamma\not\in\{a_t,b_1,c\}$ such that
  $\Gamma(s_1)(i)=\gamma$ and $\Gamma(s_2)(\gamma)=\lambda$ for some
  integer $\lambda$ in the set $\{a_t,b_1,c\}$. Since $s_2\in Z$,
  there exist elements
  $v_1,\ldots,v_{n_{s_2}}\in\{y_1,y_2,(1;b_1,a_t)\}$ such that
  $s_2=v_1\cdots v_{n_{s_2}}$. Suppose that $\gamma\in
  \{b'_2,\ldots,b'_{r'}\}\cup\{b_2,\ldots,b_r\}$. Since
  \begin{displaymath}
    \{a_1,\ldots,a_t\}\cap\{b'_1,\ldots,b'_{r'}\}=\emptyset
    \mbox{ and }
    \{a'_1,\ldots,a'_{t'}\}\cap\{b_1,\ldots,b_r\}=\emptyset,
  \end{displaymath}
  we obtain $\Gamma(v_1)(\gamma)\in
  \{b'_2,\ldots,b'_{r'}\}\cup\{b_2,\ldots,b_r\}$. Similarly, we have
  \begin{displaymath}
    \Gamma(v_1\cdots v_{n_{s_2}})(\gamma)
    \in \{b'_2,\ldots,b'_{r'}\}\cup\{b_2,\ldots,b_r\}.
  \end{displaymath}
  This contradicts the assumption that $\Gamma(s_2)(\gamma)=\lambda$
  and $\lambda\in\{a_t,b_1,c\}$. Similarly, we obtain $\gamma\not\in
  \{a'_2,\ldots,a'_{r'}\}\cup\{a_2,\ldots,a_r\}$, which entails
  $\lambda\not \in\{a_t,b_1,c\}$, a contradiction. Hence,
  $\Gamma'(\psi(s_1s_2))(i)=\theta'$ and, thus, $\psi$ is a
  homomorphism.

  The result follows.
\end{proof}

\begin{lem}
  \label{F12 prec}
  Let $T$ and $V$ be finite block groups. If $F_{12}\prec T\times V$,
  then at least one of the conditions
  $T\not\in\mathsf{Excl}(F_7,F_{12})$ or
  $V\not\in\mathsf{Excl}(F_7,F_{12})$ holds.
\end{lem}

\begin{proof}
  Since $F_{12}\prec T\times V$, there exists a subsemigroup $U$ of
  $T\times V$ such that there is an onto homomorphism
  $\phi:U\rightarrow F_{12}$. Let
  $$T\times V= R_1 \supsetneqq R_2 \supsetneqq \cdots \supsetneqq
  R_{h} \supsetneqq R_{h+1} = \emptyset$$ be a principal series of
  $T\times V$. That is, each $R_i$ is an ideal of $T\times V$ and
  there is no ideal of $T\times V$ strictly between $R_i$ and
  $R_{i+1}$. Since $T$ and $V$ are finite, there exist elements $r_1,
  r_2$ and $r_3$ in the subsemigroup $U$ and an integer $1\leq i \leq
  h$ that satisfy the following properties:
  \begin{enumerate}
  \item $\phi(r_1)=(1;2,3)$, $\phi(r_2)=w_1$ and $\phi(r_3)=w_2$;
  \item $r_1 \in R_{i} \setminus R_{i+1}$;
  \item $\phi^{-1}((1;2,3))\cap R_{i+1}=\emptyset$.
  \end{enumerate}

  By $\phi(r_1r_2r_1)=(1;2,3)$ and~(3), $R_{i} / R_{i+1}$
  is an inverse completely $0$-simple semigroup or an inverse
  completely simple semigroup, say $M=\mathcal{B}_n(G)$. Then, there
  exist an element $g\in G$ and integers $1\leq\alpha,\beta\leq n$
  such that $r_1=(g;\alpha,\beta)$. Let $o$ be the order of $G$.
  Suppose that $\alpha=\beta$. Then, we have $r_1^{o+1}=r_1$ and,
  thus, $\phi^{-1}((1;2,3))\cap \phi^{-1}(\theta)\neq \emptyset$, a
  contradiction. As $\phi(r_1r_2r_1)=\phi(r_1r^2_3r_1)=(1;2,3)$, we
  have $r_1r_2r_1,r_1r^2_3r_1\in R_{i} \setminus R_{i+1}$ and, thus,
  $\Gamma(r_1r_2r_1)=\Gamma(r_1r^2_3r_1)=(\alpha, \beta,\theta)$ where
  $\Gamma$ is an \C\ of $R/R_{i+1}$ on $M$. Therefore, $[\beta,\alpha]
  \sqsubseteq \Gamma(r_2)$ and there exists an integer
  $1\leq\gamma\leq n$ such that $[\beta,\gamma,\alpha] \sqsubseteq
  \Gamma(r_3)$. Also, since $\phi(r_1r_3r_2r_3r_1)=(1;2,3)$, we obtain
  $(\gamma) \subseteq \Gamma(r_2)$.

  The elements $r_1,r_2$ and $r_3$ are in $T\times V$. Then, there
  exist elements $t_1,t_2, t_3\in T$ and $v_1,v_2,v_3\in V$ such that
  $r_1=(t_1,v_1)$, $r_2=(t_2,v_2)$ and $r_3=(t_3,v_3)$. Let
  \begin{displaymath}
    T= T_1 \supsetneqq T_2 \supsetneqq \cdots \supsetneqq T_{h'}
    \supsetneqq T_{h'+1} = \emptyset
  \end{displaymath}
  and
  \begin{displaymath}
    V= V_1 \supsetneqq V_2 \supsetneqq \cdots \supsetneqq V_{h''}
    \supsetneqq V_{h''+1} = \emptyset
  \end{displaymath}
  be, respectively, principal series of $T$ and $V$. Suppose that $t_1
  \in T_{j} \setminus T_{j+1}$ and $v_1 \in V_{k} \setminus V_{k+1}$,
  for some integers $1\leq j\leq h'$ and $1\leq k\leq h''$.

  Since $[\alpha,\beta] \sqsubseteq \Gamma(r_1)$ and $[\beta,\alpha]
  \sqsubseteq \Gamma(r_2)$, we obtain $(r_1r_2)^or_1(r_2r_1)^o=r_1$.
  Then, we have
  \begin{eqnarray}\label{t_1v_1}
    (t_1t_2)^ot_1(t_2t_1)^o=t_1\mbox{ and }(v_1v_2)^ov_1(v_2v_1)^o=v_1.
  \end{eqnarray}
  By \cite[Theorem 3]{Fab}, we have $R_{i} \setminus R_{i+1}=(T_{j}
  \setminus T_{j+1})\times (V_{k} \setminus V_{k+1})$. Let $A_1=T_{j}
  \setminus T_{j+1}$ and $A_2=V_{k} \setminus V_{k+1}$. Again,
  by~(\ref{t_1v_1}), the subset $A_i\cup\{\theta\}$ is an inverse
  completely $0$-simple semigroup or $A_i$ is an inverse completely
  simple semigroup, say $M_{A_i}=\mathcal{B}_{n_{A_i}}(G_{A_i})$ or
  $M_{A_i}=\mathcal{B}_{1}(G_{A_i})$ for $1\leq i\leq 2$. Then, there
  exist elements $(l;a,b)\in M_{A_1}$ and $(l';a',b')\in M_{A_2}$,
  such that $t_1=(l;a,b)$ and $v_1=(l';a',b')$.

  Since $r_1^2$ is in $R_{i+1}$ and $R_{i} \setminus R_{i+1}=A_1\times
  A_2$, we obtain $t_1^2\not\in A_1$ or $v_1^2\not\in A_2$. By
  symmetry, we may assume that $t_1^2\not\in A_1$, whence $a\neq b$.
  As $r_1r_2r_1,r_1r^2_3r_1\in A_1\times A_2$, we have
  $t_1t_2t_1,t_1t^2_3t_1\in A_1$ and, thus,
  $\Gamma'(t_1t_2t_1)=\Gamma'(t_1t^2_3t_1)=(a, b,\theta)$ where
  $\Gamma'$ is an \C\ of $T/T_{j+1}$ on $M_{A_1}$. Therefore, $[b,a]
  \sqsubseteq \Gamma'(t_2)$ and there exists an integer $1\leq x\leq
  n_{A_1}$ such that $[b,x,a] \sqsubseteq \Gamma'(t_3)$. Also, since
  $r_1r_3r_2r_3r_1\in A_1\times A_2$, it follows that
  $t_1t_3t_2t_3t_1\in A_1$ and, thus, $(x) \subseteq \Gamma'(t_2)$. By
  Lemma~\ref{m,l,m'}, we conclude that $F_{7}\prec T$ or $F_{12}\prec
  T$.
\end{proof}


Lemma~\ref{F12 prec} yields the following theorem.

\begin{thm}
  The class $\mathsf{\M}$ of all finite $\M$ semigroups is a pseudovariety.
\end{thm}

%
%
%

\begin{prop}
  \label{PS \M identity}
  The sequence
  \begin{displaymath}
    \left\{\bigl(
      \lambda_{n!}(x,y,w_1,w_2,w_1,w_2,\ldots),
      \rho_{n!}(x,y,w_1,w_2,w_1,w_2,\ldots)\bigr)
      \mid n\in \mathbb{N}\right\}
  \end{displaymath}
  converges in $(\overline{\Omega}_{\{x,y,w_1,w_2\}}\mathsf{S})^2$.
\end{prop}

\begin{proof}
  The proof is similar to the proof of Proposition~\ref{PS PE
    identity}.
%
\end{proof}

We denote the limit of the sequence
\begin{displaymath}
  \bigl(\lambda_{n!}(x,y,w_1,w_2,w_1,w_2,\ldots),
  \rho_{n!}(x,y,w_1,w_2,w_1,w_2,\ldots)\bigr)_{n}
\end{displaymath}
in the preceding proposition by
$(\overline{\lambda}(x,y,w_1,w_2),\overline{\rho}(x,y,w_1,w_2))$.

\begin{lem}
  \label{\M} %
  Let $S\in \mathsf{BG_{nil}}$. The semigroup
  $S$ is not $\M$ if and only if, there exist elements $a,w_1,w_2 \in
  S$ such
  that
  \begin{displaymath}
    \overline{\lambda}(aw_2,w_2a,w_1,w_2)\neq\overline{\rho}(aw_2,w_2a,w_1,w_2)
  \end{displaymath}
  and
  \begin{align*}
    (\overline{\rho}(aw_2,w_2a,w_1,w_2) w_1&
    \overline{\lambda}(aw_2,w_2a,w_1,w_2))^{\omega+1}\\
    &\quad\quad=\overline{\rho}(aw_2,w_2a,w_1,w_2) w_1
    \overline{\lambda}(aw_2,w_2a,w_1,w_2).  
  \end{align*}
\end{lem}

\begin{proof}
  Suppose that there exist elements $a,w_1,w_2\in S$ such that
  \begin{displaymath}
    \overline{\lambda}(aw_2,w_2a,w_1,w_2)\neq\overline{\rho}(aw_2,w_2a,w_1,w_2)
  \end{displaymath}
  and
  \begin{align*}
    (\overline{\rho}(aw_2,w_2a,w_1,w_2) w_1&
    \overline{\lambda}(aw_2,w_2a,w_1,w_2))^{\omega+1}\\
    &\quad\quad=\overline{\rho}(aw_2,w_2a,w_1,w_2) w_1
    \overline{\lambda}(aw_2,w_2a,w_1,w_2).  
  \end{align*}
  Then  
  \begin{displaymath}
    \lambda_{k^2+1}(aw_2,w_2a,w_1,w_2,w_1,w_2,\ldots)
    \neq \rho_{k^2+1}(aw_2,w_2a,w_1,w_2,w_1,w_2,\ldots)
  \end{displaymath}
  where
  $k=\abs{S}$.
  By the pigeonhole principle, there exist positive integers $t$ and
  $r$ such that $t<r\leq k^2+1$ with
  \begin{align*}
    &\bigl(
      \lambda_{t}(aw_2,w_2a,w_1,w_2,w_1,w_2,\ldots),
      \rho_{t}(aw_2,w_2a,w_1,w_2,w_1,w_2,\ldots)\bigr)\\
    &= \bigl(
      \lambda_{r}(aw_2,w_2a,w_1,w_2,w_1,w_2,\ldots),
      \rho_{r}(aw_2,w_2a,w_1,w_2,w_1,w_2,\ldots)\bigr).
  \end{align*}
  Put $s_1= \lambda_{t}$, $s_2=\rho_{t}$ and $h=r-t$. To simplify the
  notation, we let $\widehat{\lambda}_j=\lambda_{t+j}$ and
  $\widehat{\rho}_j=\rho_{t+j}$ for $j\ge0$. In this notation, we may write
  $s_1 =\widehat{\lambda}_0=\widehat{\lambda}_h 
  =\lambda_{h}(s_1, s_2, v_{t+1}, \ldots, v_{t+h})$ and
  $s_2 =\widehat{\rho}_0=\widehat{\rho}_h 
  =\rho_{h}(s_1,s_2, v_{t+1}, \ldots, v_{t+h})$ where $v_{2l-1}=w_1$
  and $v_{2l}=w_2$ for all $1\leq l$. Note further that $s_1\ne s_2$.

  Let 
  \begin{displaymath}
    S= S_1 \supsetneqq S_2 \supsetneqq \cdots \supsetneqq S_{s}
    \supsetneqq S_{s+1} = \emptyset
  \end{displaymath}
  be a principal series of $S$. Suppose that $s_1 \in S_i \setminus
  S_{i+1}$ for some $1\leq i\leq s$. Because $S_i$ and $S_{i+1}$ are
  ideals of $S$, the above equalities yield $s_2 \in S_i \setminus
  S_{i+1}$ and $v_{t+1}, \ldots, v_{t+h} \in S \setminus S_{i+1}$.
  Furthermore, $S_i / S_{i+1}$ is a completely $0$-simple semigroup.

  First, suppose that $S_i \setminus S_{i+1}$ is a group. We denote by
  $e$ its identity element. Since $s_1,s_2 \in S_i \setminus S_{i+1}$,
  the equalities $s_1 =\lambda_{h}(s_1, s_2, v_{t+1}, \ldots,
  v_{t+h})$ and $s_2 =\rho_{h}(s_1, s_2, v_{t+1}, \ldots, v_{t+h})$
  yield $\widehat{\lambda}_j, \widehat{\rho}_j \in S_{i}\setminus
  S_{i+1}$, for every $1 \leq j \leq h$. Now, as $e$ is the identity
  element of $S_{i}\setminus S_{i+1}$, we have $\widehat{\lambda}_j
  v_{t+j+1} \widehat{\rho}_j=\widehat{\lambda}_j e v_{t+j+1}
  \widehat{\rho}_j$ and $\widehat{\rho}_j v_{t+j+1}
  \widehat{\lambda}_j=\widehat{\rho}_je v_{t+j+1}\widehat{\lambda}_j$,
  for every $0 \leq j < h$. If $S_{i+1} \neq \emptyset$ and
  $ev_{t+j}\in S_{i+1}$ holds for some $1 \leq j \leq h$, then
  $s_1,s_2\in S_{i+1}$, a contradiction. So, we must have $ev_{t+j}
  \in S_i \setminus S_{i+1}$, for all $1 \leq j \leq h$. Consequently,
  the following conditions are satisfied:
  \begin{align*}
    &s_1 =\lambda_{h}(s_1, s_2, v_{t+1}, \ldots, v_{t+h}) %
      =\lambda_{h}(s_1, s_2, ev_{t+1}, \ldots, ev_{t+h})\\
    &\neq s_2 = \rho_{h}(s_1,s_2, v_{t+1}, \ldots, v_{t+h}) %
      =\rho_{h}(s_1,s_2, ev_{t+1}, \ldots, ev_{t+h}).
  \end{align*}
  By Lemma~\ref{finite-nilpotent}, it follows that $S_i \setminus
  S_{i+1}$ is a non-nilpotent group. Hence, $S\not\in
  \mathsf{BG_{nil}}$, a contradiction.

  Now, suppose that $S_i \setminus S_{i+1}$ is not a group. By
  \cite[Lemma 2.1]{Jes-Okn}, in this case, we have
  $S_{i}/S_{i+1}=\mathcal{B}_n(G)$ with $G$ a nilpotent group and
  $n>1$. Also, since $s_{1},s_{2}\in S_{i}\setminus S_{i+1}$, there
  exist $1 \leq n_1,n_2,n_3,n_4 \leq n$ and $g, g' \in G$ such that
  $s_1 =(g; n_1, n_2)$ and $s_2 = (g'; n_3, n_4)$.
  Hence, 
  \begin{displaymath}
    [n_2,n_3;n_4,n_1] \sqsubseteq \Gamma(v_{t+1}),[n_2, n_1;n_4, n_3]
    \sqsubseteq \Gamma(v_{t+2}).
  \end{displaymath}

  Suppose that $(n_1, n_2)= (n_3, n_4)$. Then, there exist elements
  $k_{j} \in G$ such that $(k;\alpha,n_2)v_{t+j}=
  (kk_{j};\alpha,n_1)$, for every $k \in G, \alpha \in \{n_1,n_2\}$
  and $1 \leq j \leq h$. Since $\widehat{\lambda}_{j-1}=
  (g_{j-1};n_1,n_2)$ and $\widehat{\rho}_{j-1}= (g'_{j-1};n_1,n_2)$
  for some $g_{j-1},g_{j-1}'\in G$, we get
  $$\widehat{\lambda}_{j}= (g_{j-1}k_jg'_{j-1};n_1,n_2),~
  \widehat{\rho}_{j}= (g'_{j-1}k_jg_{j-1};n_1,n_2)$$ and, thus,
  \begin{displaymath}
    g = \lambda_{h}(g, g', k_{1}, \ldots, k_{h}),~ g' =
    \rho_{h}(g, g',k_{1}, \ldots, k_{h}).
  \end{displaymath}
  In view of Lemma~\ref{finite-nilpotent}, this yields a contradiction
  with $G$ being nilpotent. We have shown that $(n_{1},n_{2})\neq
  (n_{3},n_{4})$.

  If $t$ is an even integer,
  then
  \begin{displaymath}
    \Gamma(\overline{\lambda}(aw_2,w_2a,w_1,w_2))=(n_1,n_2,\theta),\
    \Gamma(\overline{\rho}(aw_2,w_2a,w_1,w_2))=(n_3,n_4,\theta),
  \end{displaymath}
  \begin{displaymath}
    [n_2,n_3;n_4,n_1] \sqsubseteq \Gamma(w_{1}),
    \mbox{ and }
    [n_2, n_1;n_4, n_3] \sqsubseteq
    \Gamma(w_{2}).
  \end{displaymath}
  Otherwise, we have
  \begin{displaymath}
    \Gamma(\overline{\lambda}(aw_2,w_2a,w_1,w_2))=(n_1,n_4,\theta),\
    \Gamma(\overline{\rho}(aw_2,w_2a,w_1,w_2))=(n_3,n_2,\theta),
  \end{displaymath}
  \begin{displaymath}
    [n_2, n_1;n_4, n_3] \sqsubseteq \Gamma(w_{1}),
    \mbox{ and }
    [n_2,n_3;n_4,n_1] \sqsubseteq \Gamma(w_{2}).
  \end{displaymath}
  By symmetry, we may assume that $t$ is even. Now, as
  \begin{align*}
      (\overline{\rho}(aw_2,w_2a,w_1,w_2)&w_1
    \overline{\lambda}(aw_2,w_2a,w_1,w_2))^{\omega+1}\\
    &\quad\quad=\overline{\rho}(aw_2,w_2a,w_1,w_2) w_1
    \overline{\lambda}(aw_2,w_2a,w_1,w_2),
  \end{align*}
  we obtain $n_2=n_3$. Therefore, the following relations hold:
  \begin{displaymath}
    [n_4,n_1] \sqsubseteq \Gamma(w_{1}),\
    (n_2) \subseteq \Gamma(w_{1}),
    \mbox{ and }
    [n_4,n_2, n_1] \sqsubseteq \Gamma(w_{2}).
  \end{displaymath}
  Now, by Lemmas~\ref{m,l} and \ref{m,l,m'}, we have $F_{7}\prec S$ or
  $F_{12}\prec S$. This shows that $S\not\in\mathsf{\M}$.

  Conversely, suppose that $S$ is not $\M$. It follows that $S\not\in
  \mathsf{Excl}(F_7,F_{12})$. Then, there exist a subsemigroup $U$ of
  $S$ and an onto homomorphism $\phi:U\rightarrow T$ such that $T$ is
  isomorphic with $F_7$ or $F_{12}$. Let
  \begin{displaymath}
    U= U_1 \supsetneqq U_2 \supsetneqq \cdots \supsetneqq U_{s'}
    \supsetneqq U_{s'+1} = \emptyset
  \end{displaymath}
  be a principal series of $U$. Suppose that $T$ is isomorphic with
  $F_7$. Since $U$ is finite, there exist elements $a$ and $r_1$ in
  the semigroup $U$ and an integer $1\leq i'\leq s'$ that satisfy the
  following properties:
  \begin{enumerate}
  \item $\phi(a)=(1;1,1)$ and $\phi(r_1)=u$;
  \item $a \in U_{i'} \setminus U_{i'+1}$;
  \item $\phi^{-1}((1;1,1))\cap U_{i'+1}=\emptyset$.
  \end{enumerate}
  Since
  $(\overline{\lambda}(ar_1,r_1a,r_1^2,r_1),\overline{\rho}(ar_1,r_1a,r_1^2,r_1))
  =(\lambda_{\gamma},\rho_{\gamma})$ for some even integer~$\gamma$,
  we obtain
  \begin{displaymath}
    \bigl(\phi(\overline{\lambda}(ar_1,r_1a,r_1^2,r_1)),
    \phi(\overline{\rho}(ar_1,r_1a,r_1^2,r_1))\bigr)
    =\bigl((1;1,2),(1;2,1)\bigr).
  \end{displaymath}
  It follows that
  $\overline{\lambda}(ar_1,r_1a,r_1^2,r_1)\neq\overline{\rho}(ar_1,r_1a,r_1^2,r_1)$.
  Now, as
  \begin{displaymath}
    \phi(\overline{\rho}(ar_1,r_1a,r_1^2,r_1) r_1^2
    \overline{\lambda}(ar_1,r_1a,r_1^2,r_1))=(1;2,2),
  \end{displaymath}
  similarly as in the proof of Lemma~\ref{F12 prec}, we have
  \begin{align*}
   (\overline{\rho}(ar_1,r_1a,r_1^2,r_1)& r_1^2
    \overline{\lambda}(ar_1,r_1a,r_1^2,r_1))^{\omega+1}\\
    &\quad\quad=\overline{\rho}(ar_1,r_1a,r_1^2,r_1) r_1^2
    \overline{\lambda}(ar_1,r_1a,r_1^2,r_1).
  \end{align*}
 
  Now, suppose that $T$ is isomorphic with $F_{12}$. Since $U$ is
  finite, there exist elements $a,r_1,r_2$ in the semigroup $U$ and an
  integer $1\leq i'\leq s'$ that satisfy the following properties:
  \begin{enumerate}
  \item $\phi(a)=(1;2,3)$, $\phi(r_1)=w_1$ and $\phi(r_2)=w_2$;
  \item $a \in U_{i'} \setminus U_{i'+1}$;
  \item $\phi^{-1}((1;2,3))\cap U_{i'+1}=\emptyset$.
  \end{enumerate}
  Since, $\phi(\overline{\lambda}(ar_2,r_2a,r_1,r_2))=(1;2,1)$ and
  $\phi(\overline{\rho}(ar_2,r_2a,r_1,r_2))=(1;1,3)$, we obtain
  $\overline{\lambda}(ar_2,r_2a,r_1,r_2)\neq\overline{\rho}(ar_2,r_2a,r_1,r_2)$.
  Now, as $$\phi(\overline{\rho}(ar_2,r_2a,r_1,r_2) w_1
  \overline{\lambda}(ar_2,r_2a,r_1,r_2))=(1;1,1),$$ again similarly as
  in the proof of Lemma~\ref{F12 prec}, we have
  \begin{align*}
      (\overline{\rho}(ar_1,r_1a,r_1^2,r_1)& r_1^2
    \overline{\lambda}(ar_1,r_1a,r_1^2,r_1))^{\omega+1}\\
    &\quad\quad=\overline{\rho}(ar_1,r_1a,r_1^2,r_1) r_1^2
    \overline{\lambda}(ar_1,r_1a,r_1^2,r_1).
  \end{align*}

  The result follows.
\end{proof}

By Lemma~\ref{\M}, we have $\mathsf{NT}\subseteq\mathsf{\M}$. In
Section~\ref{sec:comp-above-pseud}, we show that the pseudovariety
$\mathsf{NT}$ is strictly contained in $\mathsf{\M}$.

Lemma~\ref{\M} yields the following theorem.

\begin{thm}
  \label{PS NTL4}
  Let $S$ be a finite semigroup. The semigroup $S$ is in $\mathsf{\M}$
  if and only if $S\in \mathsf{BG_{nil}}$ and satisfies the
  implication
  \begin{align*}
    &(\overline{\rho}(aw_2,w_2a,w_1,w_2) w_1
      \overline{\lambda}(aw_2,w_2a,w_1,w_2))^{\omega+1}\\
      &\quad=\overline{\rho}(aw_2,w_2a,w_1,w_2) w_1
      \overline{\lambda}(aw_2,w_2a,w_1,w_2)\\ 
    &\qquad\qquad\qquad\Rightarrow\overline{\lambda}(aw_2,w_2a,w_1,w_2)
      =\overline{\rho}(aw_2,w_2a,w_1,w_2).
  \end{align*}
\end{thm}

Even though we have shown that $\mathsf{\M}$ is a pseudovariety, and
therefore admits a basis of pseudoidentities by Reiterman’s theorem,
we leave as an open problem to determine a simple such basis. In
particular, we do not know whether $\mathsf{\M}$ admits a finite basis
of pseudoidentities.

From \cite[Corollary 4.2]{Jes-Sha3}, it follows that every \D\
semigroup $S$ in the pseudovariety $\mathsf{\overline{G}}_{nil}$ is an
epimorphic image of one of the following semigroups:
\begin{enumerate}
\item a semigroup of right or left zeros with 2 elements;
\item $\mathcal{B}_2(G) \cup^{\Gamma} T$ such that $T=\langle u
  \rangle \cup \{\theta\}$ with $\theta$ the zero of $S$,
  $u^{2^{k}}=1$ the identity of $T \setminus \{\theta\}$ (and of $S$)
  and $\Gamma (u)= (1,2)$ (for $k=1$, one obtains $F_7$);
\item $\mathcal{B}_3(G) \cup^{\Gamma} \langle w_1, w_2\rangle$, with
  $\Gamma(w_1)=(2,1,3,\theta)$ and $\Gamma(w_2)=(2,3,\theta)(1)$,
  $w_{2}w_{1}^{2}=w_{1}^{2}w_{2}=w_{1}^{3}=w_{2}w_{1}w_{2}= \theta$;
\item $\mathcal{B}_n(G) \cup^{\Gamma} \langle v_1, v_2\rangle$, with
  \begin{displaymath}
    [k, m;k',m']\sqsubseteq \Gamma(v_1)
    \mbox{ and }
    [k, m';k', m] \sqsubseteq \Gamma(v_2)
  \end{displaymath}
  for pairwise distinct integers $k, k', m$ and $m'$ between $1$ and
  $n$, there do not exist distinct integers $l_1$ and $l_2$ between
  $1$ and $n$ such that $(l_1,l_2)\subseteq \Gamma(x)$ for some $x\in
  \langle v_1, v_2 \rangle$, and there do not exist pairwise distinct
  integers $o_1, o_2$ and $o_3$ between $1$ and $n$ such that
  \begin{displaymath}
    (o_2,o_1,o_3,\theta) \subseteq \Gamma(y_1),
    (o_2,o_3,\theta)(o_1)\subseteq \Gamma(y_2)
  \end{displaymath}
  for some $y_1,y_2\in \langle v_1, v_2 \rangle$.
\end{enumerate}

The pseudovariety $\mathsf{BG_{nil}}$ does not contain any semigroup
in (1). The pseudovariety $\mathsf{PE}$ does not contain any semigroup
in (1) and (2). The pseudovariety $\mathsf{\M}$ does not contain any
semigroup in (1), (2) and (3). Clearly the pseudovariety $\mathsf{MN}$
does not contain any semigroup in (1)--(4).

\section{The pseudovariety \pv{NT}}
\label{sec:pvy-pvnt}


Unlike the pseudovarieties $\mathsf{MN}$ and $\mathsf{PE}$, the pseudovariety $\mathsf{NT}$ turns out to have infinite rank.

In this section, we use the concept of automaton congruence. For this
reason, we first give this definition for the convenience of the
reader. Recall that a congruence on the (incomplete) automaton
$\mathcal{A}=(Q,\Sigma,\delta)$ is an equivalence relation $\sim$ on
the state set $Q$ such that if $p \sim q$ and $\delta(p,a)$ and
$\delta(q, a)$ are both defined in $\mathcal{A}$, then $\delta(p,a) \sim
\delta(q, a)$. The quotient automaton $\mathcal{A}/{\sim}$ has set of
states $Q/{\sim}$ and it has an $a$-labeled transition from the
$\sim$-class $[p]$ of $p$ to $[q]$ if there exists an $a$-labeled
transition of $\mathcal{A}$ from $p'$ to $q'$ for some states $p'\sim
p$ and $q'\sim q$.


The main result of this section is the following theorem, whose proof
occupies the remainder of the section.

\begin{thm}
  \label{PS NT}
  The pseudovariety $\mathsf{NT}$ has infinite rank and, therefore, it
  is non-finitely based.
\end{thm}

\begin{proof}
  To establish the theorem, we show that, for infinitely many positive
  integers $m$, there exists a finite semigroup $S\not\in\mathsf{NT}$
  all of whose $(m-1)$-generated subsemigroups lie in $\mathsf{NT}$.

  Consider the alphabet $A_n=\{x,y,w_1,\ldots,w_{n+2}\}$ and its
  action on the set
  \begin{displaymath}
    Q_n=\{a_i,a_i',b_i,b_i',c_i,d_i: i=1,\ldots,2^n
    \} 
  \end{displaymath}
  given by the following formulas, where we adopt common notation for
  partial transformations of a finite set, in terms of a 2-row matrix
  in which the first row is the domain and the image of an element in
  the domain is indicated immediately below it:
  \begingroup
\allowdisplaybreaks
  \begin{align*}
    \eta(x)
    &=
      \begin{pmatrix}
        c_1&c_2&c_3&\ldots&c_{2^{n}}&b'_1&b'_2&b'_3&\ldots&b'_{2^n} \\
        a_1&a_2&a_3&      &a_{2^{n}}&d_1 &d_2 &d_3 &      &d_{2^{n}}
      \end{pmatrix},\\
    \eta(y)
    &=
      \begin{pmatrix}
        a'_1&a'_2&a'_3&\ldots&a'_{2^{n}}&d_1&d_2&d_3&\ldots&d_{2^n} \\
        c_1 &c_2 &c_3 &      &c_{2^{n}} &b_1&b_2&b_3&      &b_{2^{n}}
      \end{pmatrix},\\ 
    \eta(w_1)
    &=
      \begin{pmatrix}
     a_{2^{n-1}+1}&a_{2^{n-1}+2}&\ldots&a_{2^{n}}   &   b_{2^{n-1}+1}&b_{2^{n-1}+2}&\ldots&b_{2^{n}}     \\
     b'_1         &b'_2         &      &b'_{2^{n-1}}&   a'_1         &a'_2         &      &a'_{2^{n-1}} 
      \end{pmatrix},\\
    \eta(w_2)
    &=
      \begin{pmatrix}
        b_{2^{n-2}+1} &b_{2^{n-2}+2} &\ldots&b_{2^{n-1}}          &a_{2^{n-2}+1} &\ldots&a_{2^{n-1}} \\
        b'_{2^{n-1}+1}&b'_{2^{n-1}+2}&      &b'_{2^{n-1}+2^{n-2}} &a'_{2^{n-1}+1}&      &a'_{2^{n-1}+2^{n-2}}
      \end{pmatrix},\\
    \eta( w_i)
    &=
      \left(\begin{matrix}
          a_{2^{n-i}+1}                    &a_{2^{n-i}+2}                    &\ldots&a_{2^{n-(i-1)}}\\
          b'_{2^{n-1}+\cdots+2^{n-(i-1)}+1}&b'_{2^{n-1}+\cdots+2^{n-(i-1)}+2}&      &b'_{2^{n-1}+\cdots+2^{n-i}}
        \end{matrix}\right.\\
    &\qquad\qquad
      \left.\begin{matrix}
          b_{2^{n-i}+1}                    &b_{2^{n-i}+2}                    &\ldots&b_{2^{n-(i-1)}} \\
          a'_{2^{n-1}+\cdots+2^{n-(i-1)}+1}&a'_{2^{n-1}+\cdots+2^{n-(i-1)}+2}&      &a'_{2^{n-1}+\cdots+2^{n-i}}
        \end{matrix}\right),\\
    &\ \text{for odd}\ i\in\{3,\ldots,n\},\\ \\
    \eta(w_{i})
   = &
       \left(\begin{matrix}
           b_{2^{n-i}+1}                    &b_{2^{n-i}+2}                    &\ldots&b_{2^{n-(i-1)}}            \\
           b'_{2^{n-1}+\cdots+2^{n-(i-1)}+1}&b'_{2^{n-1}+\cdots+2^{n-(i-1)}+2}&      &b'_{2^{n-1}+\cdots+2^{n-i}}
         \end{matrix}\right.\\
    &\qquad\qquad
      \left.\begin{matrix}
          a_{2^{n-i}+1}                    &a_{2^{n-i}+2}                    &\ldots&a_{2^{n-(i-1)}}\\
          a'_{2^{n-1}+\cdots+2^{n-(i-1)}+1}&a'_{2^{n-1}+\cdots+2^{n-(i-1)}+2}&      &a'_{2^{n-1}+\cdots+2^{n-i}}
        \end{matrix}\right),\\
    &\ \mbox{for even}\ i\in\{4,\ldots,n\},\\ 
    \eta(w_{n+1})
    &=
      \begin{pmatrix}
        a_{1}     &b_{1} \\
        b'_{2^{n}}&a'_{2^{n}}
      \end{pmatrix},\ 
      \eta(w_{n+2})=
      \begin{pmatrix}
        b_{1}     &a_{1} \\
        b'_{2^{n}}&a'_{2^{n}}
      \end{pmatrix}. 
  \end{align*}
  \endgroup

  Let $m> 5$ be an even integer, $n=m-4$ and $N_{2^{n}}$ be the set 
  \begin{eqnarray}
    \label{ex4}
    N_{2^{n}} = M \cup\langle
    \eta(x),\eta(y),\eta(w_1),\ldots,\eta(w_{n+2})\rangle,
  \end{eqnarray}
  where the angle brackets denote the subsemigroup of the partial
  transformation semigroup on $Q_n$ generated by the given
  transformations and $M\simeq\mathcal{B}_{6\times 2^{n}}(\{1\})$ is
  the Brandt semigroup of all at most rank~1 transformations of the
  set $Q_n$. Note that $N_{2^n}$ is a subsemigroup of the partial
  transformation semigroup on the set $Q_n$.

  We claim that $N_{2^{n}}\not\in\mathsf{NT}$ and all
  $(m-1)$-generated subsemigroups of $N_{2^{n}}$ are in $\mathsf{NT}$.
  To establish our claim, we first define the finite semigroup $N$,
  which plays an important role in this proof. In the following, we
  define the automaton $\mathcal{A}_n$, whose transition semigroup is
  the semigroup $N_{2^{n}}$. We also define a congruence $\sim$ on
  $\mathcal{A}_n$ such that the transition semigroup of the quotient
  automaton $\mathcal{A}_n/{\sim}$ is the semigroup $N$.


  Let $\mathcal{N}$ be the following automaton:
  \begin{center}
    \begin{tikzpicture}[x=1.2mm,y=1.2mm,thick,->,>=stealth',
      shorten >=1pt]
      \node [state] (c) at (0,0) {$c$}; %
      \node [state] (a) at (16,-8) {$a$}; %
      \node [state] (a') at (16,8) {$a'$}; %
      \node [state] (b) at (44,8) {$b$}; %
      \node [state] (b') at (44,-8) {$b'$}; %
      \node [state] (d) at (60,0){$d$}; %
      
      \path (c) edge[below] node {$X$} (a) (a') edge[above] node {$Y$}
      (c) (a) edge[below] node {$W_1$} (b') edge[right] node {$W_2$}
      (a') (d) edge[above] node {$Y$} (b) (b') edge[below] node {$X$}
      (d) (b) edge[above] node {$W_1$} (a') edge[left] node {$W_2$}
      (b');
    \end{tikzpicture}
  \end{center}
  and $N$ be the transition semigroup of the automaton $\mathcal{N}$.
  Denoting $\varepsilon$ the transition homomorphism
  $\{X,Y,W_1,W_2\}^+\to N$, we obtain
  \begin{equation}
    \label{eq:XYW1W2}
    \varepsilon(X)=
    \begin{pmatrix}
      c&b' \\
      a&d
    \end{pmatrix},\
    \varepsilon(Y)=
    \begin{pmatrix}
      a'&d \\
      c &b
    \end{pmatrix},\
    \varepsilon(W_1)=
    \begin{pmatrix}
      a &b\\
      b'&a'
    \end{pmatrix},\
    \varepsilon(W_2)=
    \begin{pmatrix}
      b &a\\
      b'&a'
    \end{pmatrix}.
  \end{equation}
  The image of $\varepsilon$ is
  \begin{displaymath}
    N=
    \langle
    \varepsilon(X),\varepsilon(Y),\varepsilon(W_1),\varepsilon(W_{2})
    \rangle,
  \end{displaymath}
  viewed as a subsemigroup of the partial transformation semigroup on
  the set $Q=\{a,a',b,b',c,d\}$.
  Since the image of none of the rank~2 partial bijections in~\eqref{eq:XYW1W2}
  is the domain of one of them, any product of them with at least two
  factors has rank at most~1, so that
  \begin{equation}
    \label{eq:N}
    N=\mathcal{B}_6(\{1\})\cup
    \{\varepsilon(X),\varepsilon(Y),\varepsilon(W_1),\varepsilon(W_{2})\}
  \end{equation}
  where $\mathcal{B}_6({1})$ stands here for the semigroup of partial
  bijections of the state set of rank at most one.
  
  Consider the alphabet $B_n$ consisting of the following letters:
  \begin{itemize}
  \item $x,y,w_1,\ldots,w_{n+2}$,
  \item $t_{e_i,e_{i+1}}$ ($i=1,\ldots,2^n-1$,
    $e\in\{a,a',b,b',c,d\}$),
  \item $t_{a_{2^n},a'_1}$, $t_{a'_{2^n},b_1}$, $t_{b_{2^n},b'_1}$,
    $t_{b'_{2^n},c_1}$, $t_{c_{2^n},d_1}$, $t_{d_{2^n},a_1}$.
  \end{itemize}
  We extend the transition homomorphism $\eta:A_n^+\to N_{2^n}$ to
  $B_n^+$ as follows:
  \begin{displaymath}
    \eta(t_{e,f})=
    \begin{pmatrix}
      e\\
      f
    \end{pmatrix}.
  \end{displaymath}
  Note that $\eta$ still takes its values in $N_{2^n}$.

  Let $\mathcal{A}_n=(Q_n,B_n,\delta_n)$ be the automaton where
  $\delta_n(q,g)=\Gamma(g)(q)$, for all $q\in Q_n$ and $g\in B_n$,
  $\Gamma$ being the \C\ of $N_{2^{n}}$ on $M$ given by~$\eta$. Let
  $\mathcal{A}'_n=(Q,B_n,\delta'_n)$ be the quotient of the automaton
  $\mathcal{A}_n$ obtained by identifying all states with the same
  letter, that is, forgetting the numerical indices. In other words,
  there is a quotient morphism $\phi'_n:\mathcal{A}_n\rightarrow
  \mathcal{A}'_n$ where $\phi'_n(e_i)=e$ for every $e\in Q$ and $1\leq
  i\leq 2^n$. Also, let $\Gamma'(g)(q)=\delta'_n(q,g)$, for all $q\in
  Q$ and $g\in B_n$. It is easy to verify that the transition
  semigroup of the automaton $\mathcal{A}'_n$ is the semigroup $N$ and
  $\Gamma'$ is an \C\ of $N$ on $\mathcal{B}_6(\{1\})$. Denoting
  $\psi:B_n^+\to N$ the corresponding transition homomorphism, note
  that $\psi(x)=\varepsilon(X)$, $\psi(y)=\varepsilon(Y)$,
  $\psi(w_i)=\varepsilon(W_1)$ for $i$ odd, and
  $\psi(w_i)=\varepsilon(W_2)$ for $i$ even.


  \begin{lem}\label{NN}
    If $z\in B_n^{+}\setminus A_n$ and $z'\in A_n$, then we have
    $\psi(z)\ne\psi(z')$. Moreover, the following equalities hold:
    $\psi^{-1}(\varepsilon(X))=\{x\}$,
    $\psi^{-1}(\varepsilon(Y))=\{y\}$,
    $\psi^{-1}(\varepsilon(W_1))=\{w_1,w_3,\ldots,w_{n+1}\}$, and
    $\psi^{-1}(\varepsilon(W_2))=\{w_2,\ldots,w_{n+2}\}$.
  \end{lem}

  \begin{proof}
    As was already observed, none of the transformations
    in~\eqref{eq:XYW1W2} may be written as products of two elements.
    Also, under $\psi$,
    the subset $B_n\setminus A_n$ maps to $\mathcal{B}_6(\{1\})$.
    Hence, if $z\in B_n^{+}\setminus A_n$ and $z'\in A_n$, then $z$
    and $z'$ are not identified by $\psi$.
  \end{proof}

  Let $z\in A_n$ and $z_1,z_2\in B_n^{+}\setminus A_n$. We claim that
  $\eta(z)\neq\eta(z_1z_2)$. Indeed, suppose the equality holds. Then,
  the elements $z$ and $z_1z_2$ have the same action on the states of
  the automaton $\mathcal{A}_n$. So, they have the same action on the
  states of the automaton $\mathcal{A}'_n$. However, this contradicts
  Lemma \ref{NN}. Hence, we have the following corollary.

  \begin{cor}\label{N2nN2n}
    If $z\in B_n^{+}\setminus A_n$ and $z'\in A_n$, then we have
    $\eta(z)\ne\eta(z')$.\qed
  \end{cor}

  Since
  \begin{align*}
    \left\langle\bigl\{\bigr.\right.
    &\eta(t_{e_1,e_2}),\eta(t_{e_2,e_3}),\ldots,\eta(t_{e_{2^n-1},e_{2^n}}),
      \eta(t_{a_{2^n},a'_1}),\eta(t_{a'_{2^n},b_1}),\\
    &\eta(t_{b_{2^n},b'_1}),\eta(t_{b'_{2^n},c_1}),
      \eta(t_{c_{2^n},d_1}),\eta(t_{d_{2^n},a_1})
      \mid \left.\bigl.e\in \{a,a',b,b',c,d\}\bigr\}\right\rangle=M,
  \end{align*}
  the transition semigroup of the automaton $\mathcal{A}_n$ is the
  semigroup $N_{2^n}$.

  In the following lemma, we examine the $\lambda$ and $\rho$
  sequences, which is useful for testing whether the semigroups
  $N_{2^n}$ lie in~$\mathsf{NT}$.

  \begin{lem}\label{lambda-n+3}
    Suppose that the elements
    \begin{equation}\label{line-lambda-n+3}
      \delta_n\bigl(p,\lambda_{n+3}(\alpha,\beta,1,\gamma_1,\gamma_2,\ldots,
      \gamma_{n+2})\bigr),
      \delta_n\bigl(p',\rho_{n+3}(\alpha,\beta,1,\gamma_1,\gamma_2,\ldots,
      \gamma_{n+2})\bigr) 
    \end{equation}
    are defined in the automaton $\mathcal{A}_n$, for some
    $\alpha,\beta \in B_n^+$, $\gamma_1,\gamma_2,\ldots, \gamma_{n+2}
    \in B_n^*$ and $p,p'\in Q_n$. Then, one of the following
    conditions holds:
    \begin{enumerate}[label={(\roman*)}]
    \item\label{item:lambda-n+3-1}
      $\{\psi(\alpha),\psi(\beta)\}=\{\varepsilon(X),\varepsilon(Y)\}$,
      $\psi(\gamma_1)=\psi(\gamma_3)=\cdots=\psi(\gamma_{n+1})=\varepsilon(W_1)$ and
      $\psi(\gamma_2)=\psi(\gamma_4)=\cdots=\psi(\gamma_{n+2})=\varepsilon(W_2)$;
    \item\label{item:lambda-n+3-2} $\psi(\alpha)=\psi(\beta)=
      \begin{psmallmatrix}
        e\\
        e
      \end{psmallmatrix}$ 
      and $\psi(\gamma_i)\in\{
      \begin{psmallmatrix}
        e\\
        e
      \end{psmallmatrix},1\}$
      for all $1\leq i\leq n+2$ and some element
      $e\in Q$.
    \end{enumerate}
  \end{lem}
  
  \begin{proof}
    By (\ref{line-lambda-n+3}),
    $\delta_n(p,\lambda_2(\alpha,\beta,1,\gamma_1))$ and
    $\delta_n(p',\rho_{2}(\alpha,\beta,1,\gamma_1))$ the elements are
    also defined in the automaton $\mathcal{A}_n$. It follows that
    $\delta'_n(\phi'_n(p),\lambda_2(\alpha,\beta,1,\gamma_1))$ and
    $\delta'_n(\phi'_n(p'),\rho_{2}(\alpha,\beta,1,\gamma_1))$ are
    defined in the automaton $\mathcal{A}'_n$ and, thus, the
    inequalities
    \begin{displaymath}
      \lambda_{2}\bigl(\psi(\alpha),\psi(\beta),1,\psi(\gamma_1)\bigr)
      \neq 0\ne
      \rho_{2}\bigl(\psi(\alpha),\psi(\beta),1,\psi(\gamma_1)\bigr)
    \end{displaymath}
    hold in the semigroup $N$. See the appendix for a calculation
    which would be too cumbersome to carry out by hand
    checking that one of the following conditions holds:
    \begin{enumerate}[label={(\alph*)}]
    \item\label{item:lambda-n+3-3}
      $\{\psi(\alpha),\psi(\beta)\}=\{\varepsilon(X),\varepsilon(Y)\}$
      and $\psi(\gamma_1)=\varepsilon(W_1)$;
    \item\label{item:lambda-n+3-4} $\psi(\alpha)=\psi(\beta)=
      \begin{psmallmatrix}
        e\\
        e
      \end{psmallmatrix}
      $ 
      and $\psi(\gamma_1)\in\{
      \begin{psmallmatrix}
        e\\
        e
      \end{psmallmatrix}
      ,1\}$ for some state $e\in Q$.
    \end{enumerate}
    First, suppose that Condition~\ref{item:lambda-n+3-3} holds. We
    show that \ref{item:lambda-n+3-1} also holds. Note that we have
    \begin{align*}
      &\bigl\{\lambda_{2}\bigl(\psi(\alpha),\psi(\beta),1,\psi(\gamma_1)\bigr),
        \rho_{2}\bigl(\psi(\alpha),\psi(\beta),1,\psi(\gamma_1)\bigr)\bigr\}\\
      &\quad=\{\varepsilon(XYW_1YX),\varepsilon(YXW_1XY)\}.
    \end{align*}
    By~\eqref{eq:N},
    the ranks of the elements $\varepsilon(XYW_1YX)$ and
    $\varepsilon(YXW_1XY)$ are both one and, thus, we have
    \begin{displaymath}
      \{\varepsilon(XYW_1YX),\varepsilon(YXW_1XY)\}
      =\left\{
        \begin{pmatrix}
          b'\\
          a
        \end{pmatrix},
        \begin{pmatrix}
          a'\\
          b
        \end{pmatrix}
      \right\}.
    \end{displaymath}
    Since
    $\lambda_{3}\bigl(\psi(\alpha),\psi(\beta),1,\psi(\gamma_1),\psi(\gamma_2)\bigr)
    \ne 0
    \ne\rho_{3}\bigl(\psi(\alpha),\psi(\beta),1,\psi(\gamma_1),\psi(\gamma_2)\bigr)$,
    it follows that $\psi(\gamma_2)=\varepsilon(W_2)$.
    Similarly, we have
    \begin{align*}
      &\bigl\{\lambda_{3}\bigl(\psi(\alpha),\psi(\beta),1,\psi(\gamma_1),\psi(\gamma_2)\bigr),
        \rho_{3}\bigl(\psi(\alpha),\psi(\beta),1,\psi(\gamma_1),\psi(\gamma_2)\bigr)\bigr\}\\
      &\quad=\left\{\begin{pmatrix}
          a'\\
          a
        \end{pmatrix},
      \begin{pmatrix}
        b'\\
        b
      \end{pmatrix}
      \right\},
    \end{align*}
    which yields the equality $\psi(\gamma_3)=\varepsilon(W_1)$. Also,
    similarly, we obtain the equalities
    $\psi(\gamma_5)=\psi(\gamma_7)=\cdots=\psi(\gamma_{n+1})=\varepsilon(W_1)$
    and
    $\psi(\gamma_4)=\psi(\gamma_6)=\cdots=\psi(\gamma_{n+2})=\varepsilon(W_2)$.
    
    Now, suppose that Condition~\ref{item:lambda-n+3-4} holds. We
    claim that so does~\ref{item:lambda-n+3-2}. Indeed, we then have
    \begin{displaymath}
      \bigl\{\lambda_{2}\bigl(\psi(\alpha),\psi(\beta),1,\psi(\gamma_1)\bigr),
      \rho_{2}\bigl(\psi(\alpha),\psi(\beta),1,\psi(\gamma_1)\bigr)\bigr\}
      =\left\{\begin{pmatrix}
          e\\
          e
        \end{pmatrix}
      \right\}.
    \end{displaymath}
    By~\eqref{eq:N} and \eqref{eq:XYW1W2},
    if $\Gamma'(Z)(e)=e$, for some $Z\in N$, then $Z\in
    \mathcal{B}_6(\{1\})$. Now, as
    \begin{displaymath}
      \lambda_{3}\bigl(\psi(\alpha),\psi(\beta),1,\psi(\gamma_1),\psi(\gamma_2)\bigr)
      \neq 0\ne
      \rho_{3}\bigl(\psi(\alpha),\psi(\beta),1,\psi(\gamma_1),\psi(\gamma_2)\bigr),
    \end{displaymath}
    we obtain $\psi(\gamma_2)\in\{\begin{psmallmatrix}
      e\\
      e
    \end{psmallmatrix},
    1\}$. Similarly, we have
    $\psi(\gamma_i)\in\{\begin{psmallmatrix}
      e\\
      e
    \end{psmallmatrix},1\}$ for all $3\leq i\leq n+2$.
  \end{proof}
  
  To continue our investigation, we need the following lemma, which is
  easy to prove by induction.

  \begin{lem}\label{lambda-rho}
    Let
    \begin{align*}
      \lambda_{i}
      &=\lambda_{i}(x, y, 1, w_{1}, w_{2}, \ldots,w_n,
        w_{n+1},w_{n+2},w_{n+1},w_{n+2},\ldots),\\
      \rho_{i}
      &=\rho_{i}(x, y, 1, w_{1}, w_{2}, \ldots,w_n,
        w_{n+1},w_{n+2},w_{n+1},w_{n+2},\ldots).
    \end{align*}
    Then, the following equalities hold:
    \begin{align*}
      \bigl(\eta(\lambda_{1}),\eta(\rho_{1})\bigr)
      &=\left(
        \begin{pmatrix}
          b'_1&b'_2&\ldots&b'_{2^{n}}\\
          b_1 &b_2 &      &b_{2^n}
        \end{pmatrix},
    \begin{pmatrix}
      a'_1&a'_2&\ldots&a'_{2^{n}}\\
      a_1 &a_2 &      &a_{2^n}
    \end{pmatrix}
    \right),\\
      \bigl(\eta(\lambda_{i}),\eta(\rho_{i})\bigr)
      &=\bigg(
        \begin{pmatrix}
          b'_{2^{n-1}+\cdots+2^{n-(i-1)}+1}&\ldots&b'_{2^n}\\
          a_1                              &      &a_{2^{n-(i-1)}}
        \end{pmatrix},\\
      &\qquad
        \begin{pmatrix}
          a'_{2^{n-1}+\cdots+2^{n-(i-1)}+1}&\ldots&a'_{2^n}\\
          b_1                              &      &b_{2^{n-(i-1)}}
        \end{pmatrix}
    \bigg),\\
      &\qquad\mbox{for even}\ 2\leq i\leq n-2,\\
      \bigl(\eta(\lambda_{i}),\eta(\rho_{i})\bigr)
      &=\bigg(
        \begin{pmatrix}
          b'_{2^{n-1}+\cdots+2^{n-(i-1)}+1}&\ldots&b'_{2^n}\\
          b_1                          &      &b_{2^{n-(i-1)}}
        \end{pmatrix},\\
      &\qquad
        \begin{pmatrix}
          a'_{2^{n-1}+\cdots+2^{n-(i-1)}+1}&\ldots&a'_{2^n}\\
          a_1                           &      &a_{2^{n-(i-1)}}
        \end{pmatrix}
    \bigg),\\
      &\qquad\mbox{for odd}\ 3\leq i\leq n-1,\\
      \bigl(\eta(\lambda_{n}),\eta(\rho_{n})\bigr)
      &=\left(
        \begin{pmatrix}
          b'_{2^n-1}&b'_{2^n}\\
          a_{1}     &a_{2}
        \end{pmatrix},
       \begin{pmatrix}
         a'_{2^n-1}&a'_{2^n}\\
         b_{1}     &b_{2}
       \end{pmatrix}
    \right),\\
      \bigl(\eta(\lambda_i),\eta(\rho_i)\bigr)
      &=\left(
        \begin{pmatrix}
          b'_{2^n}\\
          b_{1}
        \end{pmatrix},
      \begin{pmatrix}
        a'_{2^n}\\
        a_{1}
      \end{pmatrix}
  \right),\ \mbox{for odd}\ n< i,\ \mbox{and}\\
      \bigl(\eta(\lambda_{i}),\eta(\rho_{i})\bigr)
      &=\left(
        \begin{pmatrix}
          b'_{2^n}\\
          a_{1}
        \end{pmatrix},
      \begin{pmatrix}
        a'_{2^n}\\
        b_{1}
      \end{pmatrix}
      \right),\ \mbox{for even}\ n< i.
    \end{align*}
  \end{lem}
  
  By Lemma \ref{lambda-rho}, we conclude that the subsemigroup
  $\langle\eta(A_n)\rangle$
  of $N_{2^n}$ is not in $\mathsf{NT}$.
  Also, by Corollary \ref{N2nN2n}, no subsemigroup of $N_{2^n}$
  containing $\eta(A_n)$ may be generated by less than $m=n+4$
  elements. To proceed, we prove that
  \begin{equation}
    \begin{split}
      \parbox[t]{.8\textwidth}{if $U$ is a subsemigroup
        of $N_{2^n}$ and $U$ is not in $\mathsf{NT}$, then $U$ contains
        the set $\eta(A_n)$.}
    \end{split}
    \label{U}
  \end{equation} 
  This shows that all $(m-1)$-generated subsemigroups of $N_{2^{n}}$
  are in $\mathsf{NT}$ and, thereby, (\ref{U}) is enough to complete the
  proof of Theorem~\ref{PS NT}.

  So, suppose that $U\notin\mathsf{NT}$ is a subsemigroup
  of~$N_{2^n}$. Since $U\not\in\mathsf{NT}$, there exist letters $r_1,
  r_2\in B_n^+$ and $v_{1}, v_{2}, \ldots,v_{n+2}\in B_n^*$, such that
  \begin{equation}
    \label{eq:U-non-NT-1}
    \eta(r_1), \eta(r_2), \eta(v_{1}),\ldots, \eta(v_{n+2})\in U,
  \end{equation}
  \begin{equation}
    \label{eq:U-non-NT-2}
    \parbox[t]{.8\textwidth}{the elements
      $\lambda_{n+3}\bigl(\eta(r_1), \eta(r_2), 1, \eta(v_{1})
      \eta(v_{2}), \ldots, \eta(v_{n+2})\bigr)$ and
      $\rho_{n+3}\bigl(\eta(r_1), \eta(r_2), 1, \eta(v_{1}),
      \eta(v_{2}), \ldots, \eta(v_{n+2})\bigr)$
      of~$U$ are nonzero and distinct.}
  \end{equation}
  Therefore, there exist elements
  $f,g,f',g'\in Q$ and integers $1\leq i_1,i_2,j_1,j_2\leq 2^n$ such
  that
  $\delta_n(f_{i_1},\lambda_{n+3})=g_{j_1}$,
  $\delta_n(f'_{i_2},\rho_{n+3})=g'_{j_2}$ and if
  $\delta_n(f_{i_1},\rho_{n+3})$ is defined then
  $\delta_n(f_{i_1},\rho_{n+3})\neq g_{j_1}$,
  where
  \begin{displaymath}
    \lambda_{k}=\lambda_{k}(r_1, r_2, 1, v_{1}, v_{2}, \ldots, v_{k-1})\ \text{and}\
    \rho_{k}=\rho_{k}(r_1, r_2, 1, v_{1}, v_{2}, \ldots, v_{k-1})
  \end{displaymath}
  for every $1\leq k\leq n+3$. 
  Therefore, by Lemmas \ref{NN} and \ref{lambda-n+3}, one of the
  following conditions holds:
  \begin{enumerate}[start=1,label={\bfseries (I\arabic*)}]
  \item\label{I1} we have $\{r_1,r_2\}=\{x,y\}$,
    $v_1,v_3,\ldots,v_{n+1}\in\{w_1,w_3,\ldots,w_{n+1}\}$, and
    $v_2,v_4,\ldots,v_{n+2}\in\{w_2,w_4,\ldots,w_{n+2}\}$;
  \item\label{I2} there exists an element $e\in Q$ such that, if
    $\delta_n(s,f)=t$ for some $f\in \{r_1,r_2,v_1,\ldots, v_{n+2}\}$
    and some $s,t\in Q_n$ with $f\neq 1$, then there exist integers
    $1\leq l_1,l_2\leq 2^n$ such that $s=e_{l_1}$ and $t=e_{l_2}$.
  \end{enumerate} 

  First, we assume that Case \ref{I1} holds. Suppose that there exists
  the least integer $k\in\{1,\ldots,n+2\}$ such that $v_k\neq w_k$.
  Then, for $1\leq i\leq k$, the sequences $\eta(\lambda_i)$ and
  $\eta(\rho_i)$ are given by Lemma \ref{lambda-rho}. Since $v_k\in
  \{w_1,w_2,\ldots,w_{n+2}\}$, there exists an integer $k'$ such that
  $v_k=w_{k'}$. Also, since
  $v_1,v_3,\ldots,v_{n+1}\in\{w_1,w_3,\ldots,w_{n+1}\}$ and
  $v_2,v_4,\ldots,v_{n+2}\in\{w_2,w_4,\ldots,w_{n+2}\}$, we have
  $k'\not\in \{k-1,k+1\}$. We claim that it is impossible that $k\ne
  k'$, which contradicts the assumptions $w_k\ne v_k=w_{k'}$. To
  establish the claim, we distinguish several cases, taking into
  account the choice of $\eta(w_k)$ at the beginning of the proof of
  Theorem~\ref{PS NT}. In each case, we show that
  $\eta(\lambda_{k+1})=0$ or $\eta(\lambda_{k+2})=0$, which
  contradicts our assumption that $\eta(\lambda_{n+3})\neq 0$
  (see~\eqref{eq:U-non-NT-2}). For that purpose, we take into account
  without further reference the definition of $\eta(w_{k'})$ and
  Lemma~\ref{lambda-rho}. Here are the cases in question:
  \begin{itemize}
  \item $k'< k\leq n$: since $k'< k$, we have $2^{n-(k-1)}<
    2^{n-k'}+1$; we deduce that $\eta(\lambda_{k+1})= 0$.
  \item $k'< k= n+1$: since $k'< n+1$, we have $1< 2^{n-k'}+1$ and,
    thus, we get $\eta(\lambda_{n+2})= 0$.
  \item $k'< k= n+2$: if $k'< n+1$, then we have $1< 2^{n-k'}+1$ so
    that $\eta(\lambda_{n+3})= 0$; also, if $k'= n+1$, since $k'\neq
    k-1$, then we have $\eta(\lambda_{n+3})= 0$.
  \item $k < k' \le n$:
    since $k < k'$, we have
    \begin{align*}
      \lefteqn{\quad\bigl(\eta(\lambda_{k+1}),\eta(\rho_{k+1})\bigr)}
      \\
      &\qquad=\Bigg(
        \begin{pmatrix}
          e_{2^{n-1}+\cdots+2^{n-(k-1)}+2^{n-k'}+1}&\ldots&e_{2^{n-1}+\cdots+2^{n-(k-1)}+2^{n-(k'-1)}}\\
          f_{2^{n-k}+\cdots+2^{n-(k'-1)}+1}        &
          &f_{2^{n-k}+\cdots+2^{n-k'}}
        \end{pmatrix},\\
      &\qquad\quad\quad
        \begin{pmatrix}
          g_{2^{n-1}+\cdots+2^{n-(k-1)}+2^{n-k'}+1}&\ldots&g_{2^{n-1}+\cdots+2^{n-(k-1)}+2^{n-(k'-1)}}\\
          h_{2^{n-k}+\cdots+2^{n-(k'-1)}+1}        &      &h_{2^{n-k}+\cdots+2^{n-k'}}
        \end{pmatrix}
       \Bigg),
    \end{align*}
    for some elements $e,f,g,h\in\{a,a',b,b'\}$. Since $k'\neq k+1$,
    the inequalities $2^{n-k}+1< 2^{n-k}+\cdots+2^{n-(k'-1)}+1$ and
    $2^{n-k}+\cdots+2^{n-k'}<2^{n-(k-1)}$ hold. There exists an integer $k''$ such that
    $v_{k+1}=w_{k''}$.
    We have 
    \begin{align*}
      \lefteqn{\quad\quad \eta(\lambda_{k+1}w_{k''})}
      \\
      &\qquad=
        \begin{pmatrix}
          e_{2^{n-1}+\cdots+2^{n-(k-1)}+2^{n-k'}+1}&\ldots&e_{2^{n-1}+\cdots+2^{n-(k-1)}+2^{n-(k'-1)}}\\
          f_{2^{n-k}+\cdots+2^{n-(k'-1)}+1}        &
          &f_{2^{n-k}+\cdots+2^{n-k'}}
        \end{pmatrix}\\
      &\qquad\qquad\left(\begin{matrix}
          a''_{2^{n-k''}+1}                    &\ldots&a''_{2^{n-(k''-1)}}\\
          b'_{2^{n-1}+\cdots+2^{n-(k''-1)}+1}  &      &b'_{2^{n-1}+\cdots+2^{n-k''}}
        \end{matrix}\right.\\
    &\qquad\qquad\qquad\qquad\qquad\qquad
      \left.\begin{matrix}
          b''_{2^{n-k''}+1}                    &\ldots&b''_{2^{n-(k''-1)}} \\
          a'_{2^{n-1}+\cdots+2^{n-(k''-1)}+1}  &      &a'_{2^{n-1}+\cdots+2^{n-k''}}
        \end{matrix}\right),
    \end{align*}
    for some $a'',b''\in\{a,b\}$. If $k<k''$, then
    $2^{n-(k''-1)}<2^{n-k}+1< 2^{n-k}+\cdots+2^{n-(k'-1)}+1$ which
    entails $\eta(\lambda_{k+1}w_{k''})=0$. Also, if $k''<k$, then
    $2^{n-k}+\cdots+2^{n-k'}<2^{n-(k-1)}<2^{n-k''}+1$ and, thus,
    again, we obtain $\eta(\lambda_{k+1}w_{k''})=0$. We conclude that
    $k''=k$ and $v_{k+1}=w_k$. It follows that
    \begin{align*}
      \eta(\lambda_{k+1}w_k)
      =&\left(\begin{matrix}
          e_{2^{n-1}+\cdots+2^{n-(k-1)}+2^{n-k'}+1}                       &\ldots\\
          q_{2^{n-1}+\cdots+2^{n-(k-1)}+2^{n-(k+1)}+\cdots+2^{n-(k'-1)}+1}&      
        \end{matrix}\right.\\
       &\qquad\qquad\quad\left.\begin{matrix}
           e_{2^{n-1}+\cdots+2^{n-(k-1)}+2^{n-(k'-1)}}\\
           q_{2^{n-1}+\cdots+2^{n-(k-1)}+2^{n-(k+1)}+\cdots+2^{n-k'}}
         \end{matrix}\right),
    \end{align*}
    for some element $q\in\{a,a',b,b'\}$. 
    Since $k'\neq k+1$, we have
    \begin{align*}
      &2^{n-1}+\cdots+2^{n-(k-1)}+2^{n-(k'-1)}\\
      &< 2^{n-1}+\cdots+2^{n-(k-1)}+2^{n-(k+1)}+\cdots+2^{n-(k'-1)}+1,
    \end{align*}
    and, thus, $\eta(\lambda_{k+2})= 0$.
  \item $k< k'=n+1$:
    since $k< k'$, we have $k\leq n$, which yields the equality
    \begin{align*}
      \bigl(\eta(\lambda_{k+1}),\eta(\rho_{k+1})\bigr)
      =\Bigg(&
               \begin{pmatrix}
                 e_{2^{n-1}+\cdots+2^{n-(k-1)}+1}\\
                 f_{2^{n-(k-1)}}                 
               \end{pmatrix},
      \begin{pmatrix}
        g_{2^{n-1}+\cdots+2^{n-(k-1)}+1}\\
        h_{2^{n-(k-1)}}
      \end{pmatrix}
      \Bigg),
    \end{align*}
    for some elements $e,f,g,h\in\{a,a',b,b'\}$. There exists an integer $k''$ such that
    $v_{k+1}=w_{k''}$.
    We have 
    \begin{align*}
      \lefteqn{\quad\quad \eta(\lambda_{k+1}w_{k''})}
      \\
      &\qquad=
        \begin{pmatrix}
                 e_{2^{n-1}+\cdots+2^{n-(k-1)}+1}\\
                 f_{2^{n-(k-1)}}                 
               \end{pmatrix}\left(\begin{matrix}
          a''_{2^{n-k''}+1}                    &\ldots&a''_{2^{n-(k''-1)}}\\
          b'_{2^{n-1}+\cdots+2^{n-(k''-1)}+1}  &      &b'_{2^{n-1}+\cdots+2^{n-k''}}
        \end{matrix}\right.\\
    &\qquad\qquad\qquad\qquad\qquad\qquad\qquad
      \left.\begin{matrix}
          b''_{2^{n-k''}+1}                    &\ldots&b''_{2^{n-(k''-1)}} \\
          a'_{2^{n-1}+\cdots+2^{n-(k''-1)}+1}  &      &a'_{2^{n-1}+\cdots+2^{n-k''}}
        \end{matrix}\right),
    \end{align*}
    for some $a'',b''\in\{a,b\}$. If $k<k''$, then
    $2^{n-(k''-1)}<2^{n-(k-1)}$, which yields the equality
    $\eta(\lambda_{k+1}w_{k''})=0$. Also, if $k''<k$, then
    $2^{n-(k-1)}<2^{n-k''}+1$ and, thus, again, we obtain
    $\eta(\lambda_{k+1}w_{k''})=0$. This shows that $k''=k$, so that
    $v_{k+1}=w_k$. It follows that
    \begin{align*}
      \lefteqn{\eta(\lambda_{k+1}w_k\rho_{k+1})}\\
      &\quad=\begin{pmatrix}
        e_{2^{n-1}+\cdots+2^{n-(k-1)}+1}\\
        f_{2^{n-(k-1)}}                 
      \end{pmatrix}
      \begin{pmatrix}
        e'_{2^{n-(k-1)}}           &g'_{2^{n-(k-1)}}\\
        f'_{2^{n-1}+\cdots+2^{n-k}}&h'_{2^{n-1}+\cdots+2^{n-k}}
      \end{pmatrix}
      \\ &\qquad\qquad 
           \begin{pmatrix}
             g_{2^{n-1}+\cdots+2^{n-(k-1)}+1}\\
             h_{2^{n-(k-1)}}
           \end{pmatrix}
      = 0,
    \end{align*}
    for some letters $e',f',g',h'\in\{a,a',b,b'\}$, which shows that
    $\eta(\lambda_{k+2})= 0$.
  \item $k<k'=n+2$:
    this case is handled similarly to the preceding one.
  \end{itemize}
  Hence, we must have 
  \begin{equation}
  v_1=w_1,\ldots, v_{n+2}=w_{n+2}.
  \label{v_1=w_1}
  \end{equation}
  Now, in Case~\ref{I1}, we
  have $\{r_1,r_2\}=\{x,y\}$ and (\ref{v_1=w_1}) holds. It follows from~(\ref{eq:U-non-NT-1})
  that the subsemigroup $U$ contains $\eta(A_n)$, which
  establishes~\eqref{U}. 
  
  To complete the proof of Theorem~\ref{PS NT}, we proceed to consider
  Case~\ref{I2}. It turns out to be impossible as we show that it
  leads to a contradiction.

  Now, we assume that Case \ref{I2} holds.
  Let $s\in \{r_1,r_2,v_1,\ldots,v_{n+2}\}$. 
  Taking into account how $s$ acts in the quotient automaton
  $\mathcal{A}'_n$, where it labels a loop at the state $e$, we deduce
  that either there is a factorization
  \begin{equation}
    s=\xi_1\zeta_1\xi_2\cdots\zeta_{n_s}\xi_{n_s+1}
    \label{eq:factorization}
  \end{equation}
  satisfying the
  following conditions:
  \begin{enumerate}
  \item $1\leq n_s$;
  \item $\xi_1, \xi_{n_s+1}\in \langle x,y\rangle^1$;
  \item $\xi_2,\ldots, \xi_{n_{s}}\in \langle x,y\rangle$;
  \item $\zeta_1,\ldots, \zeta_{n_{s}}\in \{w_1,\ldots,w_{n+2}\}$,
  \end{enumerate}
  or a letter of $B_n\setminus A_n$ intervenes in $s$, or $s=1$.
  
  To continue our treatment of Case \ref{I2}, we need to consider
  another quotient of the automaton $\mathcal{A}_n$, this time
  dropping the letters and retaining just the indices of the states.
  Let $I_n=\{1,\ldots,2^n\}$ and let $\mathcal{A}''_n=(I_n,B_n,\mu_n)$
  be the quotient of the automaton $\mathcal{A}_n$ given by the
  morphism $\phi''_n:\mathcal{A}_n\rightarrow \mathcal{A}''_n$ where
  $\phi''_n(e_i)=i$
  for every $e\in Q$ and $i\in I_n$. Let $\kappa_n$ be transition
  homomorphism for the automaton~$\mathcal{A}''_n$.

  \begin{lem}\label{A''}
    There exists a permutation $\sigma\in S_{2^n}$ such that
    \begin{align*}
      \kappa_n(w_i)
      &=
        \begin{pmatrix}
          \sigma(2^{i-1})  &\sigma(2^{i-1}+2^i)  &\ldots&\sigma(2^{i-1}+(2^{n-i}-1)\times 2^{i})\\
          \sigma(2^{i-1}+1)&\sigma(2^{i-1}+2^i+1)&      &\sigma(2^{i-1}+(2^{n-i}-1)\times 2^i+1)
        \end{pmatrix},\\
      \kappa_n(w_{n})
      &=\begin{pmatrix}
        \sigma(2^{n-1})\\
        \sigma(2^{n-1}+1)
      \end{pmatrix},\ 
      \kappa_n(w_{n+1})=\kappa_n(w_{n+2})=
      \begin{pmatrix}
        \sigma(2^{n})\\
        \sigma(1)
      \end{pmatrix},\\
    \end{align*} 
    for every integer $1\leq i< n$.
  \end{lem}

  \begin{proof}
    We proceed by induction on $n$.
    
    For $n=1$, let $\sigma\in S_2$ with $\sigma(1)=2$ and $\sigma(2)=1$.
    Since the equalities $\kappa_1(w_1)=\begin{psmallmatrix}
      2\\
      1
    \end{psmallmatrix}$ and
    $\kappa_1(w_2)=\kappa_1(w_3)=\begin{psmallmatrix}
      1\\
      2
    \end{psmallmatrix}$ hold, the result is
    satisfied for $n=1$.

    Now, we suppose that $\sigma\in S_{2^n}$ is a permutation such
    that the statement of the lemma holds for~$n$.
    We define a permutation $\sigma'\in S_{2^{n+1}}$ as follows:
    $\sigma'(2k)=\sigma(k)$ and $\sigma'(2k-1)=\sigma(k)+2^n$, for all
    $1\leq k\leq 2^n$. To establish that $\sigma'$ has the required
    property for the next value of $n$, we need to relate the actions
    $\kappa_{n+1}(w_{i+1})$ of $w_{i+1}$ in $\mathcal{A}''_{n+1}$ and
    $\kappa_n(w_i)$ of $w_i$ in $\mathcal{A}''_n$:
    \begin{align*}
      \kappa_{n+1}(w_{i+1})
      &=\left(\begin{matrix}
          2^{n-i}+1                       &2^{n-i}+2                     &\ldots    \\
          2^n+2^{n-1}+\cdots+2^{n-(i-1)}+1 &2^n+2^{n-1}+\cdots+2^{n-(i-1)}+2      
        \end{matrix}\right.\\
      &\qquad\qquad\qquad\qquad\qquad\qquad\qquad\qquad\qquad
        \left.\begin{matrix}
            2^{n-(i-1)}   \\
            2^n+2^{n-1}+\cdots+2^{n-i}
          \end{matrix}\right),\\
      \kappa_n(w_{i})
      &=\left(\begin{matrix}
          2^{n-i}+1                   &2^{n-i}+2                   &\ldots\\
          2^{n-1}+\cdots+2^{n-(i-1)}+1 &2^{n-1}+\cdots+2^{n-(i-1)}+2&   
        \end{matrix}\right.\\   
      &\qquad\qquad\qquad\qquad\qquad\qquad\qquad\qquad\qquad\qquad
        \left.\begin{matrix}
            2^{n-(i-1)}\\
            2^{n-1}+\cdots+2^{n-i}
          \end{matrix}\right).
    \end{align*}
    Then, from the induction hypothesis, we get
    \begin{align*}
      \kappa_{n+1}(w_{i+1})
      &=\begin{pmatrix}
        \sigma(2^{i-1})      &\ldots&\sigma(2^{i-1}+(2^{n-i}-1)\times 2^{i})    \\
        \sigma(2^{i-1}+1)+2^n&      &\sigma(2^{i-1}+(2^{n-i}-1)\times 2^i+1)+2^n
      \end{pmatrix}\\
      &=\begin{pmatrix}
        \sigma'(2^{i})  &\ldots&\sigma'(2^{i}+(2^{(n+1)-(i+1)}-1)\times 2^{i+1})\\
        \sigma'(2^{i}+1)&      &\sigma'(2^{i}+(2^{(n+1)-(i+1)}-1)\times 2^{(i+1)}+1)
      \end{pmatrix},
    \end{align*} 
    for $1\leq i\leq n$. 

    Also, since $\kappa_{n+1}(w_1)=\begin{psmallmatrix}
      2^{n}+1&\ldots&2^{n+1}\\
      1 & &2^{n}
    \end{psmallmatrix}$ and $1\leq \sigma(1),\ldots,\sigma(2^n)\leq 2^n$, we have
    \begin{align*} 
      \kappa_{n+1}(w_1)
      &=\begin{pmatrix}
        \sigma(1)+2^n&\sigma(2)+2^n&\ldots&\sigma(2^n)+2^n\\
        \sigma(1)    &\sigma(2)    &      &\sigma(2^n)
      \end{pmatrix}\\
      &=\begin{pmatrix}
        \sigma'(1)&\sigma'(3)&\ldots&\sigma'(2^{n+1}-1)\\
        \sigma'(2)&\sigma'(4)&      &\sigma(2^{n+1})
      \end{pmatrix}.
    \end{align*} 
    Similarly, we have
    $\kappa_{n+1}(w_{n+2})=\kappa_{n+1}(w_{n+3})=
    \begin{psmallmatrix}
      \sigma'(2^{n+1})\\
      \sigma'(1)
    \end{psmallmatrix}$.
    The result follows.
  \end{proof}

  By Lemma \ref{A''}, the automaton $\mathcal{A}''_n$ is a cycle. For
  example, the automaton $\mathcal{A}''_4$ is drawn in
  Figure~\ref{fig:Adoubleprime4}.
  \begin{figure}[ht]
    \begin{center}\scriptsize
      \begin{tikzpicture}[x=1.2mm,y=1.2mm,thick,->,>=stealth',shorten
        >=1pt]
        \node [state] (1) at (48,20) {$\sigma(1)$}; 
        \node [state] (2) at (63,13) {$\sigma(2)$}; 
        \node [state] (3) at (75,03) {$\sigma(3)$}; 
        \node [state] (4) at (80,-12) {$\sigma(4)$}; 
        \node [state] (5) at (80,-28) {$\sigma(5)$}; 
        \node [state] (6) at (75,-45) {$\sigma(6)$}; 
        \node [state] (7) at (63,-55) {$\sigma(7)$}; 
        \node [state] (8) at (48,-60) {$\sigma(8)$}; 
        \node [state] (9) at (31,-60) {$\sigma(9)$}; 
        \node [state] (10) at(16,-55) {$\sigma(10)$}; 
        \node [state] (11) at(5,-45)  {$\sigma(11)$}; 
        \node [state] (12) at(0,-28)  {$\sigma(12)$}; 
        \node [state] (13) at(0,-12)  {$\sigma(13)$}; 
        \node [state] (14) at(5,03)   {$\sigma(14)$}; 
        \node [state] (15) at(16,13)  {$\sigma(15)$}; 
        \node [state] (16) at(31,20)  {$\sigma(16)$}; 
        
        \path (1) edge[above] node {$w_1$} (2) 
        (2) edge[above] node {$\ w_2$} (3) 
        (3) edge[right] node {$w_1$} (4) 
        (4) edge[right] node {$w_3$} (5) 
        (5) edge[right] node {$w_1$} (6) 
        (6) edge[right]  node {$\ w_2$} (7) 
        (7) edge[below] node {$\quad w_1$} (8) 
        (8) edge[below] node {$w_4$} (9) 
        (9) edge[below] node {$w_1$} (10)
        (10) edge[below] node {$w_2\quad$} (11)
        (11) edge[left] node {$w_1$} (12)
        (12) edge[left] node {$w_3$} (13)
        (13) edge[left] node {$w_1$} (14)
        (14) edge[above] node {$w_2\quad$} (15)
        (15) edge[above] node {$w_1\quad$} (16)
        (16) edge[above] node {$w_5,w_6$} (1);
      \end{tikzpicture}
    \end{center}
    \caption{The automaton $\mathcal{A}''_4$}
    \label{fig:Adoubleprime4}
  \end{figure}

  We come back to Case \ref{I2}. Let $\sigma\in S_{2^n}$ be the
  permutation given by Lemma \ref{A''} for the automaton
  $\mathcal{A}''_n$ and $\sigma^{-1}$ be the inverse of the
  permutation~$\sigma$. Since
  $\delta_n(f_{i_1},\lambda_{n+3})=g_{j_1}$ and
  $\delta_n(f'_{i_2},\rho_{n+3})=g'_{j_2}$, the condition of
  Case~\ref{I2} implies that there exist integers $1\leq
  l_1,l_2,l_3,l_4\leq 2^n$ such that
  $\delta_n(e_{l_1},\lambda_n)=e_{l_2}$ and
  $\delta_n(e_{l_3},\rho_n)=e_{l_4}$. Hence, we have
  $\mu_n(l_1,\lambda_n)=l_2$ and $\mu_n(l_3,\rho_n)=l_4$. Now, we
  claim that
  \begin{equation}
    \eta(\lambda_n)=\begin{pmatrix}
      e_{l_1}\\
      e_{l_2}
    \end{pmatrix}
    \text{ and }
    \eta(\rho_n)= \begin{pmatrix}
      e_{l_3}\\
      e_{l_4}
    \end{pmatrix}.
    \label{eq:lamdan-rhon-rank-1}
  \end{equation}
  If a letter of $B_n\setminus A_n$ intervenes in any element of the
  set $\{r_1,r_2,v_1,\ldots,v_{n-1}\}$ then the claim clearly holds.
  Now, suppose that no letter of $B_n\setminus A_n$ intervenes in any
  element of the set $\{r_1,r_2,v_1,\ldots,v_{n-1}\}$. Since
  $r_1,r_2\neq 1$, we may consider the numbers $\nu(r_i)$ of factors
  $w_j$ ($1\le j\le n+2$) that intervene in the corresponding
  factorization~\eqref{eq:factorization} of~$r_i$. Note that
  $\nu(r_i)>0$ because $n_{r_i}\ge1$.
  Hence, the numbers of occurrences of elements of the set
  $\{w_1,\ldots,w_{n+2}\}$ as factors in factorizations of $\lambda_n$
  and $\rho_n$ are at least~$2^n$. We have $\mu_n(i,z)=i$, for every
  integer $1\leq i\leq 2^n$ and an element $z\in\{x,y\}$. Also, if
  $\mu_n(i,z)$ is defined for some integer $1\leq i\leq 2^n$ and
  element $z\in\{w_1,\ldots,w_{n+2}\}$, then
  $\sigma^{-1}(\mu_n(i,z))=\sigma^{-1}(i)+1$. Now, as
  $\mu_n(l_1,\lambda_n)$ and $\mu_n(l_3,\rho_n)$ are defined and the
  automaton $\mathcal{A}''_n$ is a cycle of length $2^n$, the
  sequences $\lambda_n$ and $\rho_n$ go through the cycle of
  $\mathcal{A}''_n$ at least once. Therefore, we have
  \begin{align*}
      &\lambda_n=\alpha_1w_{n+1}\beta_1,\ \rho_n=\alpha_2w_{n+1}\beta_2 \mbox{ or}\\
      &\lambda_n=\alpha_1w_{n+2}\beta_1,\ \rho_n=\alpha_2w_{n+2}\beta_2,
  \end{align*}
  for some $\alpha_1,\beta_1,\alpha_2,\beta_2\in A_n^*$. Let
  $\nu(\alpha_1)$ be the numbers of factors $w_j$ ($1\le j\le n+2$)
  that intervene in $\alpha_1$. Suppose that there exist integers
  $1\leq l'_1,l'_2\leq 2^n$ such that
  $\delta_n(e_{l'_1},\lambda_n)=e_{l'_2}$. Since
  $\lambda_n=\alpha_1w_{n+1}\beta_1$ or
  $\lambda_n=\alpha_1w_{n+2}\beta_1$, we have
  $\mu_n(l_1,\alpha_1)=\sigma(2^n)$ and
  $\mu_n(l'_1,\alpha_1)=\sigma(2^n)$. It follows that
  $l_1=l'_1=\sigma(2^n-\nu(\alpha_1)\mod 2^n)$. Similarly, one may
  show that $l_2=l'_2$. We deduce that
  $\eta(\lambda_n)=\begin{psmallmatrix}
    e_{l_1}\\
    e_{l_2}
  \end{psmallmatrix}$. Also, in the same way, one shows that
  $\eta(\rho_n)=\begin{psmallmatrix}
    e_{l_3}\\
    e_{l_4}
  \end{psmallmatrix}$.
  
  Similarly, if $v_{n}=CzD$ with $z\in\{w_n,w_{n+1},w_{n+2}\}$, then
  there exist integers $1\leq l''_1,l''_2\leq 2^n$ such that
  $\eta(v_n)=\begin{psmallmatrix}
    e_{l''_1}\\
    e_{l''_2}
  \end{psmallmatrix}$. Then, by (\ref{eq:lamdan-rhon-rank-1}), we have
  $\eta(\lambda_{n+1})=\begin{psmallmatrix}
    e_{l_1}\\
    e_{l_2}
  \end{psmallmatrix}\begin{psmallmatrix}
    e_{l''_1}\\
    e_{l''_2}
  \end{psmallmatrix}\begin{psmallmatrix}
    e_{l_3}\\
    e_{l_4}
  \end{psmallmatrix}$ and $\eta(\rho_{n+1})=\begin{psmallmatrix}
    e_{l_3}\\
    e_{l_4}
  \end{psmallmatrix}\begin{psmallmatrix}
    e_{l''_1}\\
    e_{l''_2}
  \end{psmallmatrix}
  \begin{psmallmatrix}
    e_{l_1}\\
    e_{l_2}
  \end{psmallmatrix}$. Now, as
  $\eta(\lambda_{n+1}),\eta(\rho_{n+1})\neq 0$, we have
  $l_2=l_4=l''_1$ and $l_1=l_3=l''_2$. It follows that
  $\eta(\lambda_{n+1})=\eta(\rho_{n+1})=\begin{psmallmatrix}
    e_{l_1}\\
    e_{l_2}
  \end{psmallmatrix}$. Also, if a letter of $B_n\setminus A_n$
  intervenes in $v_{n}$ or $v_n=1$, then
  $\eta(\lambda_{n+1})=\eta(\rho_{n+1})$. Thus, no letter of
  $B_n\setminus A_n$ intervenes in $v_{n}$, $v_n\neq 1$, and
  \begin{equation}
    \text{the letters $w_n,w_{n+1},w_{n+2}$ do not intervene in $v_n$.}
    \label{eq:vn-without-side-transitions}
  \end{equation}
  Suppose that the number of occurrences of letters from the set
  $\{w_1,\ldots,w_{n+2}\}$ in the factorization of $v_{n}$
  in~(\ref{eq:factorization}) is at least $2^{n-1}$. Then, there
  exists a factorization $v_n=v'_nv''_n$ such that the number of
  occurrences of letters from the set $\{w_1,\ldots,w_{n+2}\}$ in the
  factorization of $v'_{n}$ is at least $2^{n-1}$ and
  strictly less than $2^{n}-1$. Thus, there exist integers
  $q_1\in\{1,\ldots,2^{n-1}\}$ and $q_2\in \{2^{n-1}+1,\ldots,2^n\}$
  such that $\mu_n(\sigma(q_1),v'_n)=\sigma(q_2)$ or
  $\mu_n(\sigma(q_2),v'_n)=\sigma(q_1)$. Now, since
  \begin{align*}
    \kappa_n(w_{n}) &=\begin{pmatrix}
      \sigma(2^{n-1})\\
      \sigma(2^{n-1}+1)
    \end{pmatrix}\
    \text{and}\ 
      \kappa_n(w_{n+1})=\kappa_n(w_{n+2})=
      \begin{pmatrix}
        \sigma(2^{n})\\
        \sigma(1)
      \end{pmatrix},
    \end{align*} 
    at least one of the letters $w_n,w_{n+1},w_{n+2}$ intervenes in
    $v'_n$, which contradicts \eqref{eq:vn-without-side-transitions}.
    Hence, the number of occurrences of letters from the set
    $\{w_1,\ldots,w_{n+2}\}$ in the factorization of $v_{n}$
    in~(\ref{eq:factorization}) is strictly less than $2^{n-1}-1$.
  
    Let $E=\{\sigma(1),\ldots,\sigma(2^{n-1})\}$ and
    $F=\{\sigma(2^{n-1}+1),\ldots,\sigma(2^n)\}$. As
    $\lambda_{n+1}=\lambda_nv_n\rho_n$ and
    $\rho_{n+1}=\rho_nv_n\lambda_n$, from
    \eqref{eq:lamdan-rhon-rank-1} we get that
    $\delta_n(e_{l_2},v_n)=e_{l_3}$ and
    $\delta_n(e_{l_4},v_n)=e_{l_1}$, and, thus, we have
  \begin{equation}
    \text{$\mu_n(l_2,v_n)=l_3$ and $\mu_n(l_4,v_n)=l_1$.}
    \label{eq:vn}
  \end{equation}
  By~\eqref{eq:vn-without-side-transitions}, from the structure of the
  automaton $\mathcal{A}''_n$ we deduce that $l_2,l_3\in G_1$ and
  $l_4,l_1\in G_2$ for some sets $G_1,G_2\in\{E,F\}$. Now, as
  $\eta(\lambda_{n+1})=\begin{psmallmatrix}
    e_{l_1}\\
    e_{l_4}
  \end{psmallmatrix}$ and $\eta(\rho_{n+1})=\begin{psmallmatrix}
    e_{l_3}\\
    e_{l_2}
  \end{psmallmatrix}$, with the same argument as above, we have
  $l_4,l_3\in H_1$ and $l_2,l_1\in H_2$ for some sets
  $H_1,H_2\in\{E,F\}$. This shows that $G_1=G_2$. Now, as $G_1=G_2$
  and because of~\eqref{eq:vn-without-side-transitions}, we obtain
  \begin{equation}
    \sigma^{-1}(l_2),\sigma^{-1}(l_4)<\sigma^{-1}(l_3),\sigma^{-1}(l_1).
    \label{eq:order}
  \end{equation}
  Since
  \begin{equation}
    \sigma^{-1}(l_1)-\sigma^{-1}(l_4)=\sigma^{-1}(l_3)-\sigma^{-1}(l_2)
    \label{eq:equal-steps}
  \end{equation}
  and $\eta(\lambda_n)\neq \eta(\rho_n)$, we have
  \begin{equation}
    l_3\neq l_1
    \label{eq:l3l1}
  \end{equation}   
  because of~(\ref{eq:lamdan-rhon-rank-1}). By symmetry, we may assume
  that $\sigma^{-1}(l_3)<\sigma^{-1}(l_1)$, so that
  \begin{displaymath}
    \sigma^{-1}(l_2)<\sigma^{-1}(l_4)<\sigma^{-1}(l_3)<\sigma^{-1}(l_1)
  \end{displaymath}
  by~(\ref{eq:equal-steps}). 
  
  Hence, by~(\ref{eq:vn}), since there is a
  unique element of~$A_n^+$ of which none of $w_n,w_{n+1},w_{n+2}$ is
  a factor that in $\mathcal{A}''_n$ maps $l_4$ to~$l_3$, we deduce
  that there are factorizations $v_n=AB=B'C$ such that
  $\mu_n(l_2,A)=l_4$, $\mu_n(l_4,B)=\mu_n(l_4,B')=l_3$,
  $\mu_n(l_3,C)=l_1$ and there are factorizations
  \begin{equation}
    B=\tau_1\omega_1\tau_2\cdots\tau_{n_B}\omega_{n_B}\tau_{n_B+1}\
    \text{and}\
    B'=\tau'_1\omega_1\tau'_2\cdots\tau'_{n_B}\omega_{n_B}\tau'_{n_B+1}
    \label{eq:B,B'}
  \end{equation}
  satisfying the following conditions:
  \begin{enumerate}
  \item $\tau_i,\tau'_i\in \langle x,y\rangle^1$, for all $1\leq i\leq
    n_B+1$;
  \item $\omega_i\in \{w_1,\ldots,w_{n+2}\}$, for all $1\leq i\leq n_B$.
  \end{enumerate}
  
  There exists an integer $1\leq n'$ such that $2^{n'}\leq
  \sigma^{-1}(l_1) - \sigma^{-1}(l_2) < 2^{n'+1}$. Consider the set
  \begin{displaymath}
    \St=\{\sigma(\sigma^{-1}(l_2)),\sigma(\sigma^{-1}(l_2)+1),\ldots,
    \sigma(\sigma^{-1}(l_1)-1)\}.
  \end{displaymath}
  We claim that
  \begin{equation}
    \begin{split}
      \parbox[t]{.8\textwidth}{there exists an integer $n'\leq n''$
        such that the letter $w_{n''}$ only acts in~$\mathcal{A}''_n$
        on one of the states of the set $\St$.}
    \end{split}
    \label{eq:claim}
  \end{equation}
  By Lemma~\ref{A''}, we have
  \begin{align*}
    \kappa_n(w_{n'})=
    &\left(
      \begin{matrix}
        \sigma(2^{n'-1})  &\sigma(2^{n'-1}+2^{n'})  &\ldots\\
        \sigma(2^{n'-1}+1)&\sigma(2^{n'-1}+2^{n'}+1)&      
      \end{matrix}
      \right.\\
    &\qquad\qquad\qquad\qquad
      \left.
      \begin{matrix}
        \sigma(2^{n'-1}+(2^{n-n'}-1)\times 2^{n'})\\
        \sigma(2^{n'-1}+(2^{n-n'}-1)\times 2^{n'}+1)
      \end{matrix}\right)
  \end{align*}
  Then, $w_{n'}$ acts on one or two of the states in the set $\St$. If
  $w_{n'}$ only acts on one of the states of~$\St$, then the claim
  \eqref{eq:claim} holds. Now, suppose that $w_{n'}$ acts on two of
  the states of~$\St$. Then, there exist nonnegative integers $k_1$
  and $k_2$ 
  such that $k_2\leq 2^{n-n'}-1$,
  $\sigma^{-1}(l_2)+k_1=2^{n'-1}+k_2\times 2^{n'}$, and
  $\sigma(\sigma^{-1}(l_2)+k_1),\sigma(\sigma^{-1}(l_2)+k_1+2^{n'})\in
  \St$. We deduce that $\sigma(2^{n'-1}+k_2\times 2^{n'}+2^{n'-1})\in
  \St$. We have $2^{n'-1}+k_2\times 2^{n'}+2^{n'-1}=(k_2+1)\times
  2^{n'}$. There exist integers $n''$ and $k$ such that $n'< n''$,
  $0\leq k\leq 2^{n-n''}-1$, and $2^{n''-1}+k\times
  2^{n''}=(k_2+1)\times 2^{n'}$. Now, as
  \begin{align*}
    \kappa_n(w_{n''})=
    &\left(
      \begin{matrix}
        \sigma(2^{n''-1})  &\sigma(2^{n''-1}+2^{n''})  &\ldots\\
        \sigma(2^{n''-1}+1)&\sigma(2^{n''-1}+2^{n''}+1)&      
      \end{matrix}
      \right.\\
    &\qquad\qquad\qquad\qquad
      \left.
      \begin{matrix}
        \sigma(2^{n''-1}+(2^{n-n''}-1)\times 2^{n''})\\
        \sigma(2^{n''-1}+(2^{n-n''}-1)\times 2^{n''}+1)
      \end{matrix}\right),
  \end{align*}
  $w_{n''}$ acts on the state $\sigma(2^{n'-1}+k_2\times
  2^{n'}+2^{n'-1})$ which is in $\St$. Since $2^{n'}\leq
  \sigma^{-1}(l_1) - \sigma^{-1}(l_2) < 2^{n'+1}$ and $n'< n''$,
  $w_{n''}$ may not act on the states of $\St$ twice. Then, the claim
  \eqref{eq:claim} holds.
    
  Now, using \eqref{eq:claim}, we prove that $l_2=l_4$ and
  $l_1=l_3$, which is in contradiction with (\ref{eq:l3l1}).
  As $v_n=AB=B'C$ and by (\ref{eq:vn}),
  (\ref{eq:B,B'}) and \eqref{eq:claim}, the letter $w_{n''}$
  intervenes once in $B$ and $B'$ and intervenes in neither $A$
  nor~$C$. Now, as the letter $w_{n''}$ intervenes once in each of the
  words $B$ and $B'$, by (\ref{eq:B,B'}) and \eqref{eq:claim}, there
  exists an integer $l$ such that $1\leq l\leq n_B$ and $\omega_{l}=w_{n''}$.
  If $A\not\in \langle x,y\rangle^1$, as $\mu_n(l_2,A)=l_4$, then
  there is a factorization
  \begin{equation}
    A=\tau''_1\omega'_1\tau''_2\cdots\tau''_{n_A}\omega'_{n_A}\tau''_{n_A+1}
<    \label{eq:A}
  \end{equation}
  satisfying the following conditions:
  \begin{enumerate}
  \item $0<n_A$;
  \item $\tau''_i\in \langle x,y\rangle^1$, for all $1\leq i\leq
    n_A+1$;
  \item $\omega'_i\in \{w_1,\ldots,w_{n+2}\}$, for all $1\leq i\leq n_A$.
  \end{enumerate} 
Since $v_n=AB=B'C$ and $\mu_n(l_2,v_n)=l_3$, the letter $w_{n''}$ acts on the states $\sigma(\sigma^{-1}(l_2)+n_A+l)$ and $\sigma(\sigma^{-1}(l_2)+l)$. A contradiction with \eqref{eq:claim}.
  It follows that $A\in \langle x,y\rangle^1$. Similarly, we have $C\in \langle x,y\rangle^1$.
  Now, as $\mu_n(l_2,A)=l_4$, $\mu_n(l_3,C)=l_1$ and
  $\mu_n(i,z)=i$, for every integer $1\leq i\leq 2^n$ and an element
  $z\in\{x,y\}$, we have $l_2=l_4$ and $l_1=l_3$.
  
  The above shows that Case \ref{I2} does not hold.
\end{proof}

\section{The pseudovariety \pv{TM}}
\label{sec:pvs-pvtm}

It is known that a finite group $G$ is in $\mathsf{TM}$ if and only if $G$ is an extension of a nilpotent group by a 2-group (see \cite{Bo-Ma, Sir}) and a finite semigroup $S$ is in $\mathsf{TM}$ if and only if the principal factors of $S$ are either null semigroups or inverse semigroups over groups that are extensions of nilpotent groups by 2-groups \cite[Corollary 10]{Jes-Ril}.

For the pseudovarieties $\mathsf{TM}$, $\mathsf{MN}^{(2)}$ and  $\mathsf{MN}^{(3)}$, we have the following result.


\begin{thm} \label{PS TM EUNNG} The following statements hold:
\begin{enumerate}
\item $\mathsf{TM}=\llbracket \phi^{\omega}(x)=\phi^{\omega}(y) \rrbracket$ where $\phi$ is the continuous endomorphism of the free profinite semigroup on $\{x,y\}$ such that $\phi(x)=xy$ and $\phi(y)=yx$.
\item $\mathsf{MN}^{(i)}=\llbracket \psi(\phi^{\omega}(x))=\psi(\phi^{\omega}(y)) \rrbracket$ where $\phi$ is the continuous endomorphism of the free profinite semigroup on $\{x,y,z,t\}$ such that $ \phi(x)=xzytyzx,\phi(y)=yzxtxzy,\phi(z)=z,\phi(t)=t$ and
$\psi$ runs over all homomorphisms from $\{x, y, z, t\}^+$ to $\{a_1,\ldots,a_i\}^+$ ($i=2,3$).
\end{enumerate}
\end{thm}


Theorem~\ref{PS TM EUNNG}(1) and Theorem~\ref{PS TM EUNNG}(2) follow easily from the definition of the pseudovarieties $\mathsf{TM}$ and $\mathsf{MN}^{(i)}$ ($i=2,3$), respectively.
Theorem~\ref{PS TM EUNNG}(1) has already been observed by the first author \cite[Corollary 6.4]{Alm3}. 

In \cite{Hig-Mar}, the semigroup $N_1$ is defined as the
$\theta$-disjoint union of the Brandt semigroup $\mathcal{B}_4(\{1\})$
with the null semigroup $\{w,v,\theta\}$,
\begin{eqnarray} \label{ex1}
  N_1 = \mathcal{B}_4(\{1\}) \cup^{\Gamma_{N_1}} \{w,v,\theta\},
\end{eqnarray}
where $\Gamma_{N_1}(w)=(2,3,\theta)(1,4,\theta)$ and
$\Gamma_{N_1}(v)=(1,3,\theta)(2,4,\theta)$. Note that $\langle
s_1,s_2,s_3\rangle \in\mathsf{MN}$, for all $s_1,s_2, s_3\in N_1$, and
$N_1\not\in\mathsf{MN}$.

Note also that there exists a finite semigroup $S$ such that for all
$s_1, s_2\in S$, $\langle s_1,s_2\rangle \in\mathsf{\M}$ but
$S\not\in\mathsf{\M}$. This holds, for example, for the semigroup
$F_{12}$.

\begin{cor}\label{PS MN Rank}
  The ranks of the pseudovarieties $\mathsf{MN}$, $\mathsf{PE}$,
  $\mathsf{\M}$ and $\mathsf{TM}$ are respectively $4$, $2$, $3$ and
  $2$.
\end{cor}

\begin{proof}
  By Theorems~\ref{PS PE} and \ref{PS TM EUNNG}(1), the ranks of the
  pseudovarieties $\mathsf{PE}$ and $\mathsf{TM}$ are at most~$2$.
  Since the semigroup $F_7$ is not in the
  pseudovariety $\mathsf{PE}$ and all subsemigroups of $F_7$ generated
  by an element of $F_7$ are in the pseudovariety $\mathsf{PE}$, the
  rank of the pseudovariety $\mathsf{PE}$ is $2$. 
  Let $S=\{a,b\}$ be a nontrivial right zero semigroup. It is easy to verify that $S\not\in \mathsf{TM}$. Since the subsemigroups $\{a\}$ and $\{b\}$ are in $\mathsf{TM}$,
  the rank of the pseudovariety $\mathsf{TM}$ is~$2$.

  By Theorem~\ref{PS MN} and Lemma~\ref{\M}, the ranks of the
  pseudovarieties $\mathsf{MN}$ and $\mathsf{\M}$ are bounded,
  respectively, by $4$ and $3$. The properties of the semigroups $N_1$
  and $F_{12}$ stated above imply that the ranks of the
  pseudovarieties $\mathsf{MN}$ and $\mathsf{\M}$ are respectively,
  $4$ and $3$.
\end{proof}

\section{Comparison of the above pseudovarieties}
\label{sec:comp-above-pseud}

In~\cite{Jes-Sha}, the upper non-nilpotent graph ${\mathcal N}_{S}$ of
a finite semigroup $S$ is investigated. The vertices of ${\mathcal
  N}_{S}$ are the elements of $S$ and there is an edge between $x$ and
$y$ if the subsemigroup generated by $x$ and $y$ is not nilpotent.
Thus, a finite semigroup $S$ belongs to $\mathsf{MN}^{(2)}$ if and
only if the graph ${\mathcal N}_S$ has no edges.
In \cite{Sha}, the collection of all such finite semigroups is denoted
$\mathsf{EUNNG}$.
By \cite[Theorem 2.6]{Jes-Sha}, we have $\mathsf{MN}^{(2)} \subseteq
\mathsf{PE}$. 
We proceed to give an example of a semigroup in $\mathsf{PE}\setminus
\mathsf{MN}^{(2)}$.

Let $N_2$ be the transition semigroup of the following automaton:
  \begin{center}
    \begin{tikzpicture}[x=4mm,y=4mm,thick,->,>=stealth',
      shorten >=1pt]
      \node [state] (1) at (0,10) {1}; %
      \node [state] (5) at (5,10) {5}; %
      \node [state] (4) at (10,10) {4}; %
      \node [state] (10) at (15,5) {10}; %
      %
      \node [state] (8) at (0,5) {8}; %
      \node [state] (2) at (0,0) {2}; %
      \node [state] (7) at (5,0) {7}; %
      \node [state] (3) at (10,0) {3}; %
      \node [state] (6) at (-5,5) {6}; %
      \node [state] (9) at (10,5) {9}; %
      
      \path
      (8) edge[right] node {$\alpha$} (2)
      (2) edge[below] node {$\alpha$} (6)
      (6) edge[below] node {$\alpha$} (1)
      (1) edge[below] node {$\alpha$} (5)
      (3) edge[below] node {$\alpha$} (7)
      (9) edge[right] node {$\alpha$} (4)
      (1) edge[right] node {$\beta$} (8)
      (7) edge[below] node {$\beta$} (2)
      (5) edge[below] node {$\beta$} (4)
      (4) edge[below] node {$\beta$} (10)
      (10) edge[right] node {$\beta$} (3)
      (3) edge[right] node {$\beta$} (9);
    \end{tikzpicture}
  \end{center}
  It amounts to a routine but lengthy calculation (see the appendix)
  to verify that the semigroup $N_2$ is a nilpotent extension of its
  unique 0-minimal ideal $M=\mathcal{B}_{10}(\{1\})$.
Hence, every principal factor of $N_2$ except $M$ is null. Also, since
the rank of the nonzero elements of $M$ is one, $N_2$ has no element
with a nontrivial cycle. It follows that
\begin{equation} 
  \begin{split}
    \parbox[t]{.8\textwidth}{$N_2$ does not have any element
      $\gamma$ such that $(a_1)(a_2)\subseteq\gamma$, for some
      distinct integers $a_1,a_2\in \{1,\ldots,10\}$.}
  \end{split}
  \label{eq:cycle}
\end{equation}
Suppose that $a,b\in N_2$ and $c\in N_2^1$. Since every principal
factor of $N_2$ is either a null semigroup or an inverse semigroup
over a trivial group, the semigroup $N_2$ is in $\mathsf{TM}$ by
\cite[Theorem 9]{Jes-Ril}. Thus, if $c=1$ then there exists an integer
$n$ such that
$\lambda_{n+2}(a,b,1,1,c,\ldots,c^n)=\rho_{n+2}(a,b,1,1,c,\ldots,c^n)$.
Now, suppose that $c\neq 1$. By (\ref{eq:cycle}), there exists an
integer $n_c$ such that
$c^{n_c}\in\{\overline{\theta},(1),\ldots,(10)\}$. If
$c^{n_c}=\overline{\theta}$, then clearly we have
$\lambda_{n_c+2}(a,b,1,1,c,\ldots,c^{n_c})=\rho_{n_c+2}(a,b,1,1,c,\ldots,c^{n_c})=\overline{\theta}$.
Now, suppose that $c^{n_c}=(v)$ and $[v', v, v''] \sqsubseteq c$ for
some integers $v,v',v''\in \{1,\ldots,10\}$. As in $N_2$ no element
has a nontrivial cycle, we get $v=v'=v''$. Then, we obtain
$(v)\subseteq c$ and, thus, the equality $c^{n'}=(v)$ holds for every
integer $n'\geq n_c$. Since $c^{n_c}$ is an idempotent and the Rees
quotient $N_2/M$ is nilpotent, we must have $c^{n_c}\in M$. It follows that
$\lambda_{n_c+2}(a,b,1,1,c,\ldots,c^{n_c}),
\rho_{n_c+2}(a,b,1,1,c,\ldots,c^{n_c})\in M$. If
\begin{equation}
  \label{eq:N2-PE}
  \lambda_{n_c+3}(a,b,1,1,c,\ldots,c^{n_c},c^{n_c+1})
  \neq \overline{\theta}\neq
  \rho_{n_c+3}(a,b,1,1,c,\ldots,c^{n_c},c^{n_c+1}),
\end{equation}
then
$\lambda_{n_c+2}(a,b,1,1,c,\ldots,c^{n_c}),\rho_{n_c+2}(a,b,1,1,c,\ldots,c^{n_c})\neq\overline{\theta}$
and, thus, there exist integers $i_1, i_2, j_1, j_2\in
\{1,\ldots,10\}$ such that
  $\lambda_{n_c+2}=
  \begin{psmallmatrix}
    i_1\\
    i_2
  \end{psmallmatrix}
$ and
$\rho_{n_c+2}=
  \begin{psmallmatrix}
    j_1\\
    j_2
  \end{psmallmatrix}
$.
Now, as $c^{n_c+1}=(v)$, we have $i_1=i_2=j_1=j_2=v$. It follows that
$$\lambda_{n_c+2}(a,b,1,1,c,\ldots,c^{n_c})=\rho_{n_c+2}(a,b,1,1,c,\ldots,c^{n_c})=(v).$$
On the other hand, if~\eqref{eq:N2-PE} fails, then we
conclude that
\begin{displaymath}
  \lambda_{n_c+4}(a,b,1,1,c,\ldots,c^{n_c+2})
  =\rho_{n_c+4}(a,b,1,1,c,\ldots,c^{n_c+2})
  =\overline{\theta}.
\end{displaymath}
Therefore, the semigroup $N_2$ is in $\mathsf{PE}$.

Now, let
$x=
\begin{psmallmatrix}
  4 \\
  1
\end{psmallmatrix}
$,
$y=
\begin{psmallmatrix}
  2 \\
  3
\end{psmallmatrix}
$,
$w_1=\beta\alpha=
\begin{psmallmatrix}
  1 & 3 & 10 & 7\\
  2 & 4 & 7  & 6
\end{psmallmatrix}
$, and
$w_2=\alpha\beta=
\begin{psmallmatrix}
  1 & 3 & 6 & 9\\
  4 & 2 & 8 & 10
\end{psmallmatrix}
$.
Since
\begin{displaymath}
  \lambda_k(x,y,w_1,w_2,w_1,w_2,\ldots)\neq \rho_k(x,y,w_1,w_2,w_1,w_2,\ldots)
\end{displaymath}
for every $1\leq k$, by \cite[Corollary 12]{Jes-Ril} the semigroup
$N_2$ is not nilpotent. Now, as $N_2$ is two-generated, we have
$N_2\in\mathsf{PE}\setminus\mathsf{MN}^{(2)}$. It amounts to a routine
but very long calculation (see the appendix) to verify that the
semigroup $N_2$ satisfies the identity
$\lambda_{4}(x,y,1,w_1,w_2,w_3)=\rho_{4}(x,y,1,w_1,w_2,w_3)$, for
$x,y\in N_2$ and $w_1,w_2,w_3\in N_2^1$. Hence, we have $N_2 \in
\mathsf{NT}$.

Figure~\ref{fig:interval-MN-BG} represents the strict inclusion relationships
between pseudovarieties we have considered so far.
\begin{figure}[ht]
\begin{center}
  \psscalebox{1.0 1.0} 
  {
    \begin{pspicture}(0.4,-1.8785938)(7,1.8785938)
      \rput[bl](0.0,1.6){$\mathsf{BG}$}
      \rput[bl](0.0,0.5){$\mathsf{TM}$}
      \rput[bl](0.0,-0.6){$\mathsf{BG}_{nil}$}
      \rput[bl](0.0,-1.7){$\mathsf{PE}$}
      \rput[bl](3,-0.65){$\mathsf{MN}^{(3)}$}
      \rput[bl](3,0.5){$\mathsf{MN}^{(2)}$}
      \rput[bl](6.6,-0.6){$\mathsf{NT}$}
      \rput[bl](6.5,0.5){$\mathsf{\M}$}
      \rput[bl](4.9,-1.8){$\mathsf{MN}$}
      \rput[bl](4.9,1.6){$\mathsf{PE}$} \psline[linecolor=black,
      linewidth=0.04](0.2,1.5)(0.2,0.9) \psline[linecolor=black,
      linewidth=0.04](0.2,0.4)(0.2,-0.2) \psline[linecolor=black,
      linewidth=0.04](0.2,-0.7)(0.2,-1.3) \psline[linecolor=black,
      linewidth=0.04](6.8,0.4)(6.8,-0.2) \psline[linecolor=black,
      linewidth=0.04](5.2,1.5)(6.7,0.9) \psline[linecolor=black,
      linewidth=0.04](6.7,-0.7)(5.2,-1.4) \psline[linecolor=black,
      linewidth=0.04](5.0,1.5)(3.5,0.85) \psline[linecolor=black,
      linewidth=0.04](5.0,-1.4)(3.7,-0.7) \psline[linecolor=black,
      linewidth=0.04](3.3,.4)(3.3,-0.3) \psline[linecolor=black,
      linewidth=0.04](6.4,0.6)(4.2,-0.4)
    \end{pspicture}
  }
\end{center}
  \label{fig:interval-MN-BG}
  \caption{Two intervals of a poset of pseudovarieties}
\end{figure}

We have $F_{12}\in\mathsf{MN}^{(2)}\setminus \mathsf{\M}$ and
$N_2\in\mathsf{NT}\setminus\mathsf{MN}^{(2)}$.

Also, the class $(\mathsf{MN}^{(2)}\cap\mathsf{\M})\setminus
\mathsf{NT}$ is not empty. An example of a semigroup in this class is
the subsemigroup $N_3$ of the full transformation semigroup on the set
$\{1,2,3,4\} \cup \{\theta\}$ given by the union of the semigroup
$N_1$ with the transformation $q=(2,1,\theta)(4,3,\theta)$. It is
proved in \cite{Jes-Sha} that $N_3\in\mathsf{MN}^{(2)}\setminus
\mathsf{NT}$. It is again routine (see the appendix) to verify that $N_3$ satisfies the
identity $\lambda_2(aw_2,w_2a,w_1,w_2)=\rho_2(aw_2,w_2a,w_1,w_2)$, for $a\in N_3$ and $w_1,w_2\in N_3^1$.
Now, as $N_3\in \mathsf{BG_{nil}}$, by Theorem~\ref{PS NTL4} the
semigroup $N_3$ is $\M$. 

The semigroup $F_{12}$ is not nilpotent and it is generated by 3
elements. Now, as $N_1\in \mathsf{MN}^{(3)}\setminus \mathsf{MN}$ and
$F_{12}\in \mathsf{MN}^{(2)}\setminus \mathsf{MN}^{(3)}$, we have
\begin{displaymath}
  \mathsf{MN}\subsetneqq\mathsf{MN}^{(3)}\subsetneqq\mathsf{MN}^{(2)}.
\end{displaymath}
By Lemma \ref{\M}, we have $\mathsf{MN}^{(3)}\subseteq \mathsf{\M}$.
Since $N_2 \not\in \mathsf{MN}^{(2)}$, $\mathsf{MN}^{(3)}$ is strictly
contained in $\mathsf{\M}$. Similarly, we have $N_3\in
\mathsf{MN}^{(3)}$. Hence, the pseudovarieties $\mathsf{MN}^{(3)}$ and
$\mathsf{NT}$ are incomparable.

Higgins and Margolis \cite{Hig-Mar} showed that
$N_1\in\langle\mathsf{Inv}\rangle\cap\mathsf{A}$ (in fact, they showed
that $N_1\in\langle\mathsf{Inv}\rangle\cap\mathsf{A}\setminus
\langle\mathsf{Inv}\cap\mathsf{A}\rangle$).
The semigroup $N_1$ is not in the pseudovariety
$\mathsf{MN}\cap\mathsf{A}$. We claim that the class
$\mathsf{MN}\cap\mathsf{A}\setminus \langle\mathsf{Inv}\rangle\not
\neq \emptyset$ and we conclude that the pseudovarieties
$\langle\mathsf{Inv}\rangle\cap\mathsf{A}$ and $\mathsf{MN}\cap
\mathsf{A}$ are incomparable.

Since the pseudovariety generated by finite inverse semigroups is
precisely the pseudovariety of finite semigroups whose idempotents commute \cite{Ash}, we need to show that the class $\mathsf{MN}\cap\mathsf{A}\setminus\llbracket
x^{\omega}y^{\omega}=y^{\omega}x^{\omega}\rrbracket$ is not empty. 
Let $N_4=\{e, f, ef, 0\}$ be a semigroup where $e^2 = e, f^2 = f$ and
$fe = 0\ne ef$. Since $N_4$ is $\mathcal{J}$-trivial, we have $N_4\in \mathsf{MN}\cap\mathsf{A}$. Now, as $ef \neq 0$ and $fe=0$, we have $N_4\in \mathsf{MN}\cap\mathsf{A}\setminus\llbracket x^{\omega}y^{\omega}=y^{\omega}x^{\omega}\rrbracket$. 

%

\begin{thm}\label{A MN}
  Figure~\ref{fig:aperiodic}
  represents all relations of equality and strict containment between
  the pseudovarieties represented in it:
  \begin{figure}[ht]
\begin{center}
  \psscalebox{1.0 1.0} 
  {
    \begin{pspicture}(1.5,-2.4)(5,2.6)
      \rput[bl](2.18,2.1){$\mathsf{PE}\cap\mathsf{A}
        =\mathsf{BG}\cap\mathsf{A}$}
      \rput[bl](0.4,0.9){$\mathsf{MN}^{(2)}\cap\mathsf{A}$}
      \rput[bl](0.4,-0.2){$\mathsf{MN}^{(3)}\cap\mathsf{A}$}
      \rput[bl](4.8,-0.2){$\mathsf{NT}\cap\mathsf{A}$}
      \rput[bl](4.7,0.9){$\mathsf{\M}\cap\mathsf{A}$}
      \rput[bl](2.9,-1.4){$\mathsf{MN}\cap\mathsf{A}$}
      \psline[linecolor=black, linewidth=0.04](5.2,0.8)(5.2,0.2)
      \psline[linecolor=black, linewidth=0.04](3.6,1.9)(5.1,1.3)
      \psline[linecolor=black, linewidth=0.04](5.1,-0.3)(3.6,-1)
      \rput[bl](2.9,-2.3){$\langle\mathsf{Inv}\cap\mathsf{A}\rangle$}
      \psline[linecolor=black, linewidth=0.04](3.5,-1.46)(3.5,-1.9)
      \psline[linecolor=black, linewidth=0.04](3.4,1.9)(1.6,1.2)
      \psline[linecolor=black, linewidth=0.04](1.6,-0.3)(3.4,-1)
      \psline[linecolor=black, linewidth=0.04](1.4,0.8)(1.4,0.2)
      \psline[linecolor=black, linewidth=0.04](4.7,1.1)(2.2,0)
    \end{pspicture}
  }
\end{center}
    \caption{A poset of aperiodic pseudovarieties}
    \label{fig:aperiodic}
  \end{figure}
\end{thm}


\begin{proof}
%
%

Suppose that $T\in(\mathsf{Inv}\cap\mathsf{A}) \setminus(\mathsf{MN}\cap\mathsf{A})$. 
Since $T$ is finite, $T$ has a principal series
$$T= T_1 \supsetneqq T_2 \supsetneqq \cdots \supsetneqq T_{t} \supsetneqq T_{t+1} = \emptyset .$$ 
Each principal factor $T_i /
T_{i+1}$ $(1 \leq i\leq t)$ of $T$ is aperiodic and either completely
$0$-simple or completely simple.

Since $T\not\in \mathsf{MN}$, by Lemma~\ref{finite-nilpotentW1W2} there exist a positive integer $h$, distinct elements $t_1,\, t_2 \in T$ and elements $w_1, w_2\in T$ such that
\begin{displaymath}
  t_1 =\lambda_{h}(t_1, t_2, w_1, w_2,w_1, w_2,\ldots)
  \mbox{ and }
  t_2 =\rho_{h}(t_1, t_2, w_1, w_2,w_1, w_2,\ldots).
\end{displaymath}
Suppose that $t_1 \in T_i \setminus T_{i+1}$, for some $1\leq i\leq
t$. Because $T_i$ and $T_{i+1}$ are ideals of~$T$, the above
equalities imply that $t_2 \in T_i \setminus T_{i+1}$ and $w_{1},
w_{2} \in T \setminus T_{i+1}$. Furthermore, one
obtains that $T_i / T_{i+1}$ is an aperiodic completely $0$-simple
semigroup.

If $T_i \setminus T_{i+1}$ is a trivial group, then we have a
contradiction with the assumption that $t_1$ and $t_2$ are distinct.



Now, as $T_i/T_{i+1}$ is an aperiodic completely simple inverse
semigroup, it is easy to show that it is of the form
$T_{i}/T_{i+1}=\mathcal{B}_n(\{1\})$ for some $n$.

Since $t_{1},t_{2}\in \mathcal{B}_n(\{1\})$, 
there exist $1 \leq n_1,n_2,n_3,n_4 \leq n$ such that $t_1 =(1; n_1,
n_2)$, $\, t_2 = (1; n_3, n_4)$. Note that,
\begin{displaymath}
  [n_2,n_3;n_4,n_1] \sqsubseteq \Gamma(w_1),\
  [n_2, n_1;n_4, n_3] \sqsubseteq \Gamma(w_2)
\end{displaymath}
where $\Gamma$ is an \C\ of $T/T_{i+1}$. Since $t_1 \neq t_2$ and
$\Gamma(w_1)$ restricted to $\{ 1, \ldots, n\} \setminus
\Gamma(w_1)^{-1}(\theta)$ is injective, we obtain $n_1\neq n_3$. It
follows that the cycle $(n_1,n_3)$ is contained in
$\Gamma(w_2^{-1}w_1)$. This contradicts the assumption that $T$ is
aperiodic. Hence, $\mathsf{Inv}\cap\mathsf{A}$ is contained in
$\mathsf{MN}$. The semigroup $N_4$ is in the subclass
$\mathsf{MN}\cap\mathsf{A}\setminus
\langle\mathsf{Inv}\cap\mathsf{A}\rangle$. Therefore,
$\langle\mathsf{Inv}\cap\mathsf{A}\rangle$ is strictly contained in
$\mathsf{MN}\cap\mathsf{A}$.

A finite aperiodic semigroup $S$ is in $\mathsf{TM}$ (respectively
$\mathsf{PE}$) if and only if the principal factors of $S$ are either
null semigroups or inverse semigroups \cite[Corollary 10 and
8]{Jes-Ril}. Therefore, we have
$\mathsf{PE}\cap\mathsf{A}=\mathsf{TM}\cap\mathsf{A}=\mathsf{BG}\cap\mathsf{A}$.

Since $\mathsf{MN}\subsetneqq\mathsf{MN}^{(3)}\subsetneqq\mathsf{MN}^{(2)}$, $\mathsf{MN}^{(3)}\subsetneqq \mathsf{\M}$ and the semigroups used earlier to distinguish the corresponding pseudovarieties are all aperiodic, we have $\mathsf{MN}\cap\mathsf{A}\subsetneqq\mathsf{MN}^{(3)}\cap\mathsf{A}\subsetneqq\mathsf{MN}^{(2)}\cap\mathsf{A}$ and $\mathsf{MN}^{(3)}\cap\mathsf{A}\subsetneqq\mathsf{\M}\cap\mathsf{A}$.

With the help of a computer (see the appendix), one can check that the
semigroup $N_1$ satisfies the identity
$\lambda_3(x,y,1,w_1,w_2)=\rho_3(x,y,1,w_1,w_2)$, for $x,y\in N_1$ and
$w_1,w_2\in N_1^1$. Hence, we have $N_1 \in \mathsf{NT}$. The
semigroup $N_1$ is aperiodic, but it is not in $\mathsf{MN}$, the
semigroup $N_3$ is aperiodic and in $\mathsf{\M}$ but it is not in
$\mathsf{NT}$, and the semigroup $F_{12}$ is aperiodic and in
$\mathsf{PE}$ but is not in $\mathsf{\M}$. Hence,
$\mathsf{MN}\cap\mathsf{A}$ is strictly contained in
$\mathsf{NT}\cap\mathsf{A}$, $\mathsf{NT}\cap\mathsf{A}$ is strictly
contained in $\mathsf{\M}\cap\mathsf{A}$, and
$\mathsf{\M}\cap\mathsf{A}$ is strictly contained in
$\mathsf{PE}\cap\mathsf{A}$.
\end{proof}

\begin{opm}  
Are the pseudovarieties $\mathsf{MN}^{(2)}$ and $\mathsf{MN}^{(3)}$ finitely based?
\end{opm} 


\textit{Acknowledgments.}
The authors were partially supported by CMUP, which is financed by
national funds through FCT -- Fundação para a Ciência e a Tecnologia,
I.P., under the project UIDB/00144/2020. The second author also
acknowledges FCT support through a contract based on the “Lei do
Emprego Científico” (DL 57/2016).
%

\bibliographystyle{acm}
\bibliography{ref-MN}

\section*{Appendix}

For the calculations in GAP mentioned in the paper, we start by
introducing a couple of functions which serve to test in a
transformation semigroup $S$ with zero, whether or not the equality
\begin{displaymath}
  \lambda_n(x,y,w_1,w_2,\ldots,w_n)=\rho_n(x,y,w_1,w_2,\ldots,w_n)
\end{displaymath}
holds, where $x,y\in S$ are arbitrary, and each $w_i$ satisfies the
following conditions determined by a \emph{directive vector}
$v=(v_1,\ldots,v_n)$:
\begin{itemize}
\item $w_i=1$ if $v_i=1$;
\item $w_i\in S$ is arbitrary if $v_i=\text{"+"}$
\item $w_i\in S^1$ is arbitrary if $v_i=\text{"*"}$.
\end{itemize}
There is a global control variable ``res'' (which needs to be
initialized) whose value is either true if there is no counterexample
to the identity, or the first counterexample to be found otherwise.
The function ``Malcev'' does the calculation for a given pair $x,y$,
calling itself recursively. The function ``check'' calls ``Malcev''
for all possible pairs $x,y$.

\newpage

\begin{lstlisting}[language=GAP]
res := true;
universe := function(c,S)
    if c = 1 then
        return [IdentityTransformation];
    elif c = "+" then
        return S;
    else
        return Union(S,[IdentityTransformation]);
    fi;
end;
Malcev := function(i,x,y,v,S,zero,ce)
    local z;
    if res <> true then
        return;
    fi;
    if i <= Size(v) then
        if x <> zero and y <> zero and x <> y then
            for z in universe(v[i],S) do
                Malcev(i+1,x*z*y,y*z*x,v,S,zero,Concatenation(ce,[z]));
            od;
        fi;
    else
        if x <> y then
            res := ce;
        fi;
    fi;
end;
check := function(v,S,zero)
    local x, y;
    for x in S do
        for y in S do
            Malcev(1,x,y,v,S,zero,[x,y]);
        od;
    od;
    return res;
end;
\end{lstlisting}

\newpage

\noindent
1. Testing the identity
\begin{displaymath}
  xyyx\,z\,yxxy=yxxy\,z\,xyyx
\end{displaymath}
in the semigroup $F_{12}$, where $z$ is allowed to take the value 1:
\begin{lstlisting}[language=GAP]
a := Transformation([1,4,2,4]);; 
b := Transformation([2,4,1,4]);;
c := Transformation([4,3,4,4]);;
F12 := Semigroup(a,b,c);;
zero := Transformation([4,4,4,4]);;
check([1,1,"*"],F12,zero);
\end{lstlisting}

\bigskip

\noindent
2. Determining the triples $(x,y,z)\in N\times N\times N^1$ such that
$xy\,z\,yx\ne0$ and $yx\,z\,xy\ne0$:
\begin{lstlisting}[language=GAP]
a := Transformation([2,7,4,7,7,7,7]);; 
b := Transformation([7,7,7,5,7,1,7]);;
c := Transformation([7,3,7,7,6,7,7]);;
d := Transformation([7,6,7,7,3,7,7]);; 
N := Semigroup(a,b,c,d);;
zero := Transformation([7,7,7,7,7,7,7]);;
for x in N do
    for y in N do
        xy := x*y;
        yx := y*x;
        for z in Union(N,[One(a)]) do
            if xy*z*yx <> zero and yx*z*xy <> zero then
                Print("x=",x," y=",y," z=",z,"\n");
            fi;
        od;
    od;
od;
\end{lstlisting}

\bigskip

\noindent
3. Testing the identity
\begin{displaymath}
  xyzyx\,t\,yxzxy = yxzxy\,t\,xyzyx
\end{displaymath}
in the semigroup $N_1$, where $z$ and $t$ are allowed to take the value 1:
\begin{lstlisting}[language=GAP]
a := Transformation([3,4,5,5,5]);; 
b := Transformation([4,3,5,5,5]);;
c := Transformation([5,5,1,5,5]);; 
d := Transformation([5,5,5,2,5]);;
N1 := Semigroup(a,b,c,d);;
zero := Transformation([5,5,5,5,5]);;
check([1,"*","*"],N1,zero);
\end{lstlisting}


\bigskip

\noindent
4. Testing the identity
\begin{displaymath}
  xyzyx\,t\,yxzxy\;s\;yxzxy\,t\,xyzyx
  =yxzxy\,t\,xyzyx\;s\;xyzyx\,t\,yxzxy
\end{displaymath}
in the semigroup $N_2$, where $z,t$ and $s$ are allowed to take the value 1:
\begin{lstlisting}[language=GAP]
a := Transformation([5,6,7,11,11,1,11,2,4,11,11]);;
b := Transformation([8,11,9,10,4,11,2,11,11,3,11]);;
N2 := Semigroup(a,b);;
zero := Transformation([11,11,11,11,11,11,11,11,11,11,11]);;
check([1,"*","*","*"],N2,zero);
\end{lstlisting}

\bigskip

\noindent
5. Testing the identity
\begin{displaymath}
xtztx\,t\,txzxt = txzxt\,t\,xtztx
\end{displaymath}
in the semigroup $N_3$, where $z$ and $t$ are allowed to take the
value 1 (although, for the latter, the value 1 always verifies the
identity), does not fit in our general recursive scheme so we use an
ad hoc program:
\begin{lstlisting}[language=GAP]
a := Transformation([3,4,5,5,5]);; 
b := Transformation([4,3,5,5,5]);;
c := Transformation([5,5,1,5,5]);; 
d := Transformation([5,5,5,2,5]);;
e := Transformation([5,1,5,3,5]);;
N3 := Semigroup(a,b,c,d,e);;
zero := Transformation([5,5,5,5,5]);;
control := true;;
for x in N3 do
    for t in N3 do
        xt := x*t;
        tx := t*x;
        if xt <> zero and tx <> zero and xt <> tx then
            for z in Union(N3,[One(a)]) do
                xtztx := xt*z*tx;
                txzxt := tx*z*xt;
                if xtztx*t*txzxt <> txzxt*t*xtztx then
                    Print("the identity fails with:\n");
                    Print("x=",x," z=",z," t=",t,"\n");
                    control := false;
                    break;
                fi;
            od;
        fi;
    od;
od;
if control = true then
   Print("\nthe identity holds\n");
fi;
\end{lstlisting}

\end{document}